\documentclass[12pt]{amsart}
\usepackage{graphicx, xcolor}
\usepackage{amssymb}
\numberwithin{equation}{section}

\def\N{\mathbb{N}}
\def\R{\mathbb{R}}\def\C{\mathbb{C}}

\newtheorem{theorem}{Theorem}[section]
\newtheorem{proposition}[theorem]{Proposition}

\newtheorem{example}[theorem]{Example}
\newtheorem{remark}[theorem]{Remark}
\newtheorem{ans}[theorem]{Definition}

\title{ Slow motion of internal shock layers for the Jin-Xin system in one space dimension}

\begin{document}

\maketitle

\begin{center}
MARTA STRANI\footnote{ Universit\'e Paris-Diderot (Paris 7), Institut de Math\'ematiques de Jussieu, E-mail address: \texttt{strani@dma.ens.fr}, \texttt{martastrani@gmail.com} }
\end{center}
\vskip1cm

\begin{abstract}
This paper considers the slow motion of the shock layer exhibited by the solution to the initial-boundary value problem for a scalar hyperbolic system with relaxation. Such behavior, known as {\it metastable dynamics}, is related to the presence of a first small eigenvalue for the linearized operator around an equilibrium state; as a consequence, the time-dependent solution approaches its steady state in an asymptotically  exponentially long time interval.
 In this contest, both rigorous and asymptotic approaches are used to analyze such slow motion for the Jin-Xin system. 
To describe this dynamics, we derive an  ODE for the position of the internal transition layer, proving how it drifts towards the equilibrium location with a speed rate that is exponentially slow. These analytical results are also validated by numerical computations.
\end{abstract}

\begin{quote}\footnotesize\baselineskip 14pt 
{\bf Key words.} 
Metastability, slow motion, internal layers, relaxation systems. 
 \vskip.15cm
\end{quote}
%\begin{keywords} 
%Metastability;
%slow motion;
%spectral analysis;
%viscous conservation laws.
%\end{keywords}

\pagestyle{myheadings}
\thispagestyle{plain}
\markboth{M. STRANI}{SLOW MOTION OF INTERNAL SHOCK LAYERS FOR THE JIN-XIN SYSTEM}

\section{Introduction}

The slow motion of internal shock layers has been recently widely studied. Such phenomenon, known as \textit {metastability}, is usually related to the presence of a first small eigenvalue for the linearized operator around a given equilibrium state. 
From a general point of view, a metastable behavior appears when solutions exhibit a first time scale in which they are close to some non-stationary state for an exponentially long time before converging to their asymptotic limit. As a consequence, two different time scales emerge: a first transient phase where a pattern of internal shock layers is formed in a $\mathcal O(1)$ time scale, and a subsequent exponentially slow motion where the layers drift toward their asymptotic limits.

A large class of evolution PDEs, concerning many different areas, exhibits this behavior. Among others, we include viscous shock problems (see \cite{LafoOMal94}, \cite{ReynWard95}, \cite{SunWard99}), phase transition problems described by the Allen-Cahn equation, with the fundamental contributions \cite{CarrPego89}, \cite{FuscHale89}, and Cahn-Hilliard equation, studied in \cite{AlikBatFus91} and \cite{Pego89}.

In this paper we mean to study the slow motion of the shock layer for the scalar hyperbolic system with relaxation
\begin{equation}\label{JX}
 \left\{\begin{aligned}
& \partial_t u +\partial_x v=0, \\
&\partial_t v+a^2 \partial_x \phi(u)=\frac{1}{\varepsilon} (f(u)-v), \quad \phi'(u)>0, 
  \end{aligned}\right.
\end{equation}
where the space variable $x$ belongs to a one-dimensional interval $I=(-\ell,\ell)$, $\ell >0$. 
System \eqref{JX} is a particular case of a class of more general hyperbolic relaxation systems of the form
\begin{equation*}
\partial_t \left ( \begin{aligned}
u \\
v
\end{aligned} \right)
+
\partial_x \left ( \begin{aligned}
& g(u,v) \\
& h(u,v)
\end{aligned} \right)=
\left( \begin{aligned}
& 0 \\
\varepsilon^{-1}& q(u,v),
\end{aligned}\right)
\end{equation*}
usually utilized to model a variety of non equilibrium processes in continuum mechanics: for example, non-thermal equilibrium gas dynamics (\cite{Lev96}, \cite{MulRug98}), traffic dynamics (\cite{AwRas00}, \cite{Li00}, \cite{LigWhit55}), and multiphase flows (\cite{BerGarADePVaz97}, \cite{BarVin80}, \cite{NatTes99}). Here $\varepsilon$ is a positive parameter, usually small, determining relaxation time.

In the case of system \eqref{JX}, the parameter $\varepsilon$ can be seen as a  \textit{viscosity coefficient}; we are interested in studying the behavior of the solution to \eqref{JX}  in the limit of small $\varepsilon$, and we want to identify the role of this parameter in the appearance and/or disappearance of phenomena of metastability.  

\vskip0.25cm
The main example we have in mind is the initial-boundary value problem for the quasilinear Jin-Xin system in the bounded interval $I=(-\ell,\ell)$, with Dirichlet boundary conditions, that is
\begin{equation}\label{JXbur}
 \left\{\begin{aligned}
& \partial_t u +\partial_x v=0,  &\qquad &x \in I, \ t \geq 0, \\
&\partial_t v+ a^2 \partial_x u=\frac{1}{\varepsilon} (f(u)-v), \\
& u(\pm \ell,t)=u_{\pm}, &\qquad &t \geq 0,\\
& u(x,0)=u_0(x), \quad v(x,0)=v_0(x) \equiv f(u_0(x)), &\qquad &x \in I,
  \end{aligned}\right.
\end{equation}
for some $\varepsilon, \ell,a>0$, $u_{\pm} \in \R$ and flux function $f$ that satisfies 
\begin{equation}\label{ipof}
f''(u) \geq c_0 >0, \quad f'(u_+) < 0 <f'(u_-), \quad f(u_+)=f(u_-).
\end{equation}
We stress that, once the boundary conditions for the function $u$ are chosen, the boundary conditions for the function $v$ are univocally determined. This model was firstly introduced in \cite{JinXin95} as a numerical scheme approximating solutions of the hyperbolic conservation law $\partial_t u+ \partial_xf(u)=0$. System \eqref{JXbur} is strictly hyperbolic, with the spectrum of the Jacobian $\sigma(dg,dh)^t$ composed by two distinct real eigenvalues $\pm a$.

In the relaxation limit ($\varepsilon \to 0^+$), system \eqref{JXbur} can be approximated to leading order by
\begin{equation}\label{epszero}
 \left\{\begin{aligned}
& \partial_t u+ \partial_x v =0,\\
 & v=f(u),
  \end{aligned}\right.
\end{equation}
that is
\begin{equation}\label{epsilonzero}
\partial_t u + \partial_x f(u)=0,
\end{equation}
together with $v=f(u)$, and complemented with boundary conditions
\begin{equation}\label{conbordo}
u(-\ell,t)=u_- \quad {\rm and} \quad u(\ell,t)=u_+.
\end{equation}
From the standard theory of entropy solutions to first-order quasi-linear equations of hyperbolic type, it is known that  \eqref{epsilonzero} admits a class of solutions, hence possible discontinuous, with speed of propagation $s$ given by the Rankine-Hugoniot condition
\begin{equation*}
s=\frac{[\![ f(u)]\!]}{[\![u]\!]}, 
\end{equation*}
where $[\![ \,\cdot\, ]\!]$ denotes the jump. Assumptions \eqref{ipof} guarantee that the jump from the value $u_-$ to the value $u_+$ is admissible if and only if $u_->u_+$, and its speed of propagation $s$ is equal to zero.  In this case,
equation \eqref{epsilonzero} admits a large class of stationary solutions satisfying the boundary conditions, given by all that piecewise constant functions in the form
\begin{equation*}
 u(x) =  \left\{\begin{aligned}
&u_- \quad x \in (-\ell,x_0), \\
&u_+ \quad x \in (x_0,\ell),
   \end{aligned}\right.
\end{equation*} 
where $x_0$ is a certain point in the interval. Hence, given $\xi \in (-\ell,\ell)$, we can construct a one-parameter family $\{ U_{{}_{\rm hyp}}(\cdot;\xi)\}$ of steady states, parametrized by $\xi$ that represents the location of the jump, and given by
\begin{equation}\label{Uhyp}
 U_{{}_{\rm hyp}}(x;\xi)= u_- \chi_{(-\ell,\xi)}(x)+ u_+ \chi_{(\xi,\ell)}(x),
\end{equation} 
where $\chi_I$ denotes the characteristic function of the interval $I$. We remark that, once $U_{{}_{\rm hyp}}(\cdot;\xi)$ is chosen, the class of stationary solutions $(U_{{}_{\rm hyp}},V_{{}_{\rm hyp}})$ for the original system \eqref{epszero} is given by the relation $V_{{}_{\rm hyp}}=f(U_{{}_{\rm hyp}})$, so that
\begin{equation}\label{Vhyp}
V_{{}_{\rm hyp}}(x;\xi)= f(u_-) \chi_{(-\ell,\xi)}(x)+ f(u_+) \chi_{(\xi,\ell)}(x).
\end{equation}
For the initial-boundary value problem \eqref{epsilonzero}-\eqref{conbordo}, it is possible to prove that every entropy solution converges in finite time to an element of the family  $\{ U_{{}_{\rm hyp}}(\cdot;\xi)\}$. This \textit{ stabilization property} has been proved for the first time in \cite{Liu78}, using the \textit{theory of generalized characteristic}, firstly introduced in \cite{Daf05}. In this framework, assumptions \eqref{ipof} on the flux function $f$ are crucial (see \cite[Theorem 6.1]{MasStr12}).
Hence, every entropy solution to the initial-boundary value problem
\begin{equation*}
\left\{\begin{aligned}
&\partial_t u+\partial_x f(u)=0, \quad v=f(u),\\
&u(\pm \ell,t)=u_\pm,
\end{aligned}\right.
\end{equation*}
converges in finite time to an element of the family $\{ U_{{}_{\rm hyp}}(\cdot;\xi),  V_{{}_{\rm hyp}}(\cdot;\xi)\}$.

\vskip0.25cm
For $\varepsilon >0$, the situation is very different. If we differentiate with respect to $x$ the second equation of \eqref{JXbur}, we obtain
\begin{equation}\label{differenziata}
u_t=\varepsilon(a^2\partial_x^2u-\partial_{tt}u)-\partial_x f(u).
\end{equation}
Thus, stationary solutions to \eqref{JXbur} solve
\begin{equation}\label{1}
a^2\varepsilon \partial_x^2u=\partial_xf(u),
\end{equation}
together with $\partial_x v=0$.  The presence of the Laplace operator has the effect that only a single stationary state is admitted (see \cite{KreissKreiss86}). As an example, we consider the case of Burgers flux, i.e. $a=1$, $f(u)=\frac{1}{2}u^2$. We can explicitly write the stationary solution for the problem \eqref{1}-\eqref{conbordo} as
\begin{equation}\label{Ustatbur}
\bar U_{bur}^\varepsilon(x)=-k\tanh{\left( \frac{kx}{2\varepsilon}\right)},
\end{equation}
where $k=k(\varepsilon, \ell, u_\pm)$ is implicitly defined by imposing the boundary conditions
Moreover, $\bar V_{bur}^\varepsilon(x)$ is defined by $\bar V_{bur}^\varepsilon(x)=f(k)$.

In the limit $\varepsilon \to 0^+$, the single steady state $(\bar U_{bur}^\varepsilon,\bar V_{bur}^\varepsilon)$ converges pointwise to $(U_{{}_{\rm hyp}}(\cdot;0),V_{{}_{\rm hyp}}(\cdot;0))$, while, 
for a class of general $f(u)$ that verify hypotheses \eqref{ipof}, the stationary solution $(\bar U^\varepsilon,\bar V^\varepsilon)$ converges pointwise to $(U_{{}_{\rm hyp}}(\cdot;\bar\xi), V_{{}_{\rm hyp}}(\cdot;\bar\xi))$, for some $\bar\xi \in I$.

Finally, the single steady state $(\bar U^\varepsilon, \bar V^\varepsilon)$ is asymptotically stable (for more details see the spectral analysis performed in Section \ref{3}), i.e. starting from an initial datum close to the equilibrium configuration, the time dependent solution  approaches the steady state for $t\to +\infty$.

\begin{figure}[ht]
\centering
\includegraphics[width=14cm,height=13cm]{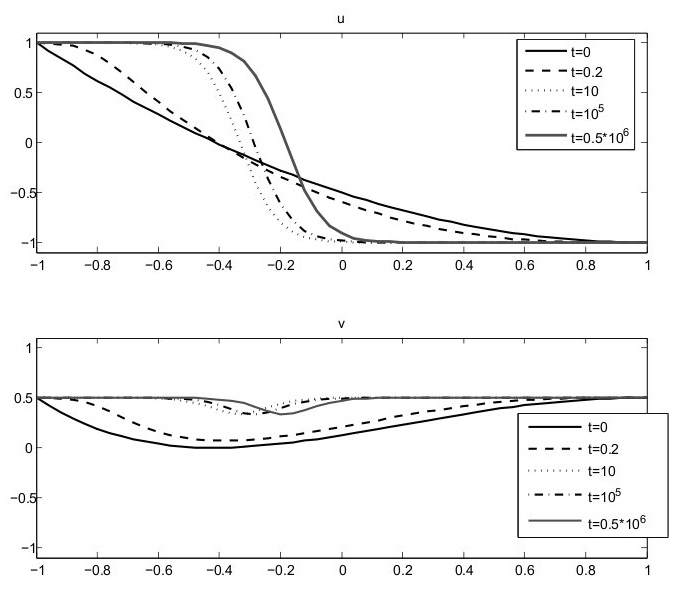}

\caption{\small{Profiles of $(u,v)$, solutions to \eqref{JXbur}, with $f(u)=u^2/2$, a =1 $\varepsilon=0.04$ and $u_{\pm}=\mp1$. The initial datum is given by the couple $(u_0(x),f(u_0(x)))$, with $u_0(x)$ a decreasing function connecting $u_+$ and $u_-$.  Profiles at times $t=0$, $0.2$, $10$, $10^5$, $0.5\times10^6$.}}\label{fig1}
\end{figure}
\vskip0.25cm
Next question is what happens to the dynamics generated by an initial datum localized far from the equilibrium solution $(\bar U^{\varepsilon},\bar V^\varepsilon)$. Numerical computations show that, starting, for example, with a decreasing initial datum $u_0(x)$ (see Fig.\ref{fig1}), because of the viscosity, a shock layer is formed in a $\mathcal O(1)$ time scale. More precisely, the solution generated by such initial datum still presents a smooth transition from $u_-$ to $u_+$, but the shock is located far away from zero, so that the solution is approximately given by a translation of the (unique) stationary solution of the problem. Once the shock layer is formed, it moves towards the equilibrium solution, and this motion is exponentially slow. Thus we have a first transient phase where the shock layer is formed, and an exponentially long time interval where the shock layer approaches the equilibrium solution.

Concerning the function $v$, starting with the initial datum $v_0(x)=f(u_0(x))$, we can observe that the position of the shock of $u$ corresponds to the location of the minimum value of the function $v$; so we have a first transient phase in which the profile of $v$ stabilizes, and an exponentially slow phase where the value of the minimum of such profile drifts towards the value $\xi$ that represents the location of the equilibrium solution for $u$.

The aim of this paper is to study the dynamics generated by an initial datum localized far from the equilibrium solution and to determine a detailed description of the low-viscosity behavior of the solutions.

\vskip0.2cm
To the best of our knowledge, the problem of the slow motion for the hyperbolic-parabolic Jin-Xin system \eqref{JX} has been never investigated before. However, system \eqref{JXbur} can be reduced by differentiation to \eqref{differenziata} together with the equation $\partial_t u+\partial_x v=0$, and, as stressed, the study of stationary solutions to \eqref{differenziata} is the same of that of the scalar conservation law
\begin{equation}\label{SCL}
\partial_t u+\partial_x f(u)= \varepsilon \partial_x^2u,
\end{equation}
together with the additional condition $\partial_x v=0$. 

A pioneering article that analyzes the dynamics of \eqref{SCL} for initial data close to the equilibrium solution has been published by G. Kreiss and H.O. Kreiss \cite{KreissKreiss86}. 
Here the flux function $f$ is given by  $f(u)=u^2/2$, so that equation \eqref{SCL} becomes the so-called \textit{viscous Burgers equation}.
To study the dynamics generated by such initial configurations, the authors consider the linearized equation close to the  unique stationary state $\bar U_{bur}^\varepsilon$ defined in \eqref{Ustatbur} , that is
\begin{equation*}
\partial_t u=\mathcal L_\varepsilon u := \varepsilon \partial_x^2u + \partial_x(b(x)u), \qquad b(x):=-f'(\bar U_{bur}^\varepsilon(x)).
\end{equation*}
In \cite{KreissKreiss86} it is shown that, if $f(u_-)=f(u_+)$, the eigenvalues of $\mathcal L_\varepsilon$ are all real and negative. Moreover
\begin{equation*}
\lambda_{1}^{\varepsilon}= \mathcal O(e^{-1/\varepsilon}) \ \ {\rm and} \ \ \lambda_k^\varepsilon \leq -c/\varepsilon <0 \quad \forall \  k \geq 2.
\end{equation*} 
This precise distribution shows that the large time behavior of solutions is described by terms of order $e^{\lambda_1^\varepsilon t}$, so that the convergence to the asymptotically stable state $\bar U_{bur}^\varepsilon(x)$ is exponentially slow, when $\varepsilon$ is small.

Concerning the phenomenon of metastability for equation \eqref{SCL}, such problem has been examined, among others, in \cite{ReynWard95} and in \cite{LafoOMal95}. Here the different approaches are based either on \textit{projection method} or on \textit{WKB expansion}, but the common aim is to derive an equation for the position of the shock layer $\xi$, considered as a function of time, that describes its slow motion towards the equilibrium location. In both papers, the analysis is carried on at a formal level and numerically validated.

A rigorous analysis has been performed firstly in \cite{deGKar98} (and generalized to the case of non-convex flux in \cite{deGKar01}). There, to study the slow motion of the internal layer, a one-parameter family of functions that approximate the stationary solution is chosen as a family of traveling waves with small velocity.

The phenomenon of metastability for the equation \eqref{SCL} has been analyzed by C. Mascia and M. Strani in \cite{MasStr12}. To study the problem of the slow motion, the authors introduce a one-parameter family of functions $\{ U^\varepsilon(\cdot,\xi)\}_{\xi \in I}$, approximating a stationary solution $\bar U^{\varepsilon}(x)$.

Hence, considering the linearized equation around $U^\varepsilon(\cdot,\xi)$, it is shown that the eigenvalues of the linearized operator verify, for all $\xi \in I$
\begin{equation}
-Ce^{-c/\varepsilon} \leq \lambda_1^\varepsilon(\xi) <0, \qquad \lambda_k^\varepsilon(\xi) \leq -C/\varepsilon \quad \forall k \geq 2.
\end{equation}
Moreover, the position of the shock layer $\xi(t)$ satisfies $|\xi(t)-\bar \xi| \leq |\xi_0|e^{-\beta^\varepsilon t}$, where $\beta^\varepsilon \sim e^{-1/\varepsilon}$.
This estimate shows that the shock layer drifts toward the equilibrium solution with a speed rate proportional to the first eigenvalue $\lambda_1^\varepsilon$, so that this motion is exponentially slow.

Motivated by the analogies among the study of our problem and some results for the scalar conservation law \eqref{SCL}, in this article we follow the approach presented in \cite{MasStr12}.

\vskip0.1cm
 -- We build-up a one parameter family of approximate steady states 
$$\{\textbf{W}^\varepsilon(x;\xi)\}_{\xi \in I}=\{ U^\varepsilon(\cdot;\xi),V^\varepsilon(\cdot;\xi)\}_{\xi \in I}$$
such that $(U^\varepsilon(\cdot;\bar \xi),V^\varepsilon(\cdot,\bar \xi)):=(\bar U^\varepsilon,\bar V^\varepsilon)$ for some $\bar \xi$, and with the additional property that $(U^\varepsilon(\cdot; \xi),V^\varepsilon(\cdot,\xi)) \to (U_{{}_{\rm hyp}}(\cdot;\xi),V_{{}_{\rm hyp}}(\cdot;\xi))$ as $\varepsilon \to 0$ in an appropriate sense. Moreover we require the error
\begin{equation*}
\left ( \begin{aligned} & \mathcal P^\varepsilon_1[\textbf{W}^\varepsilon] \\
& \mathcal P_2^\varepsilon[\textbf{W}^\varepsilon] \end{aligned} \right):= \left ( \begin{aligned}
&-\partial_x V^\varepsilon \\
-a^2 \partial_x& U^\varepsilon+\frac{1}{\varepsilon} (f(U^\varepsilon)-V^\varepsilon)
\end{aligned} \right)
\end{equation*}
to be small in $\varepsilon$ in a sense to be specified.
\vskip0.1cm
 -- We describe the dynamics of the system in a neighborhood of the family $\{ U^\varepsilon(\cdot;\xi),V^\varepsilon(\cdot;\xi)\} $.

Once a set of reference states $( U^\varepsilon(\cdot;\xi),V^\varepsilon(\cdot;\xi))$ is chosen, we determine spectral properties of the linearized operator around such element; moreover we show that, under appropriate hypotheses on how far is an element of the family of approximate steady states from being an exact stationary solution, a metastable behavior appears.

The main difference with respect to \cite{MasStr12} is that here we deal with an hyperbolic system. Hence, since the linearized operator around the reference state $\{ \textbf{W}^\varepsilon\}$  is not necessarily self-adjoint, we have to consider the chance of having complex eigenvalues, and the spectral analysis need much more care.
\vskip0.25cm
This paper is organized as follows.

In Section \ref{2} we propose a construction for the family $\{U^\varepsilon,V^\varepsilon\}$ in the case of the Jin-Xin system. Then, we write the solution as
\begin{equation*}
\left\{\begin{aligned}
u(x,t)&=u^1(x,t)+U^\varepsilon(x;\xi(t)),\\
v(x,t)&=v^1(x,t)+V^\varepsilon(x;\xi(t)).
\end{aligned}\right.
\end{equation*}
Hence we use as new coordinates the position of the shock layer $\xi$, and the perturbation $Y=(u^1,v^1)$. The couple $(\xi,Y)$ turns to solve an ODE-PDE coupled system of equations. We then study and approximation of such system, obtained by linearizing with respect to $Y$ and by keeping the nonlinear dependence on $\xi$, so that the $o(Y)$-terms can be neglected.

\vskip0.1cm
In Section \ref{3}, we analyze spectral properties of the linear operator $\mathcal L^\varepsilon_\xi$ arising from the linearization around $\{U^\varepsilon,V^\varepsilon\}$: we show that the spectrum of $\mathcal L^\varepsilon_\xi$ can be decomposed into three parts: the first eigenvalue is real, negative and $\lambda^{JX}_{1}=\mathcal{O}(e^{-C/\varepsilon})$, $C>0$, hence small as $\varepsilon \to 0^+$; all the other real eigenvalues are of order $-C/\varepsilon$; all the remaining eigenvalues are complex with real and imaginary part less than $-C/\varepsilon$, $C>0$.
Such estimates show that all of the components relative to all of the eigenvectors except the first one have a very fast decay for small $\varepsilon$, so that a slow motion occurs as a consequence of the size of the first eigenvalue.

\vskip0.1cm
In Section \ref{4}, following the idea of \cite{ReynWard95}, we give a precise asymptotic expression for the first eigenvalue $\lambda^{JX}_1$ of $\mathcal L^\varepsilon_\xi$, showing that it is exponentially small for $\varepsilon \to 0$.

\vskip0.1cm
Finally, in Section \ref{5}, by using the spectral analysis performed in the previous Sections, we analyze the system for the couple $(\xi,Y)$; our main result is Theorem \ref{mainthe}, where we prove the following estimate for the $L^2$-norm of the perturbation $Y$
\begin{equation}\label{estintro}
|Y|_{{}_{L^2}}(t) \leq  \left[C_1|\Omega_1^\varepsilon|_{{}_{L^\infty}}+C_2|\Omega^\varepsilon_2|_{{}_{L^\infty}}\right] t+e^{-\mu^\varepsilon t}|Y_0|_{L^2},
\end{equation}
for some constants $C_1$ and $C_2$ independent on $\varepsilon$, and where the terms $\mu^\varepsilon$, $\Omega_1^\varepsilon$ and $\Omega^\varepsilon_2$ are small in $\varepsilon$ in a sense that will be specified in details later on.  Precisely, the perturbation $Y$ has a very fast decay in time, up to a reminder that is bounded by $\Omega^\varepsilon_1$ and $\Omega_2^\varepsilon$, hence small in $\varepsilon$.

Estimate \eqref{estintro} can be used to decouple the system for the variables $(\xi,Y)$. This leads us to the statement of  Proposition \ref{SLBfin}, providing a precise estimate for the variable $\xi(t)$. In particular, we will show that  the shock layer position drifts towards the equilibrium location at a speed rate that becomes smaller as $\varepsilon \to 0$.

Finally, we numerically compute the position of the shock layer $\xi$ at various time, showing that numerical results agree with the analytical results.

\section{General Framework}\label{2}

Let us consider the Jin-Xin system
\begin{equation}\label{JX2}
 \left\{\begin{aligned}
& \partial_t u +\partial_x v=0,  &\qquad &x \in I, \ t \geq 0, \\
&\partial_t v+ a^2 \partial_x u=\frac{1}{\varepsilon} (f(u)-v), \\
& u(\pm \ell,t)=u_{\pm}, &\qquad &t \geq 0,\\
& u(x,0)=u_0(x), \quad v(x,0)=v_0(x) \equiv f(u_0(x)), &\qquad &x \in I,
  \end{aligned}\right.
\end{equation}
for some flux function $f$ chosen so that assumptions \eqref{ipof} hold. System \eqref{JX2} can be rewritten as
\begin{equation}\label{compactform}
\partial_t Z=\mathcal F^\varepsilon[Z], \quad Z \big|_{t=0}=Z_0,
\end{equation}
where
\begin{equation*}
Z=\left ( \begin{aligned}
u \\
v
\end{aligned} \right)
\qquad \mathcal F^\varepsilon[Z] :=\left ( \begin{aligned} & \mathcal P^\varepsilon_1[Z] \\
& \mathcal P_2^\varepsilon[Z] \end{aligned} \right)= \left ( \begin{aligned}
&-\partial_x v \\
-a^2 \partial_x& u+\frac{1}{\varepsilon} (f(u)-v)
\end{aligned} \right).
\end{equation*}
We are interested in studying the behavior of the solution to \eqref{compactform} in the relaxation limit, i.e. $\varepsilon \to 0$.  
We assume that there exists a one-parameter family of functions $$\{ U^\varepsilon(\cdot;\xi), V^\varepsilon(\cdot;\xi)\}_{\xi \in I} $$ such that $(U^\varepsilon(\cdot;\bar \xi),V^\varepsilon(\cdot;\bar \xi))=(\bar U^\varepsilon,\bar V^\varepsilon)$ for some $\bar \xi \in I$, where $(\bar U^\varepsilon,\bar V^\varepsilon)$ is the exact steady state of the system.

When $\xi \neq \bar \xi$, an element of this family can be seen as an approximate stationary solution to the problem, i.e. $\mathcal F^\varepsilon[U^\varepsilon(\cdot;\xi), V^\varepsilon(\cdot;\xi)] \to 0$ as $\varepsilon \to 0$ in an appropriate sense to be specified. Moreover we require that, in the relaxation limit,  $(U^\varepsilon(\cdot;\xi),V^\varepsilon(\cdot;\xi)) \to (U_{{}_{\rm hyp}}(\cdot; \xi), V_{{}_{\rm hyp}}(\cdot,\xi ))$, where $U_{{}_{\rm hyp}}$ and $V_{{}_{\rm hyp}}$ are defined in \eqref{Uhyp}-\eqref{Vhyp}.

Let us stress that, once the one-parameter family of functions $\{ U^\varepsilon(\cdot;\xi) \} $ is chosen, the couple $\{ U^\varepsilon(\cdot;\xi), V^\varepsilon(\cdot;\xi)\}$ is univocally determined by the relation
\begin{equation*}
 V^\varepsilon=-\varepsilon a^2 \partial_x U^\varepsilon+f(U^\varepsilon).
\end{equation*}

\begin{example}\label{ex1}\rm{

In the case of Burgers flux, i.e. $f(u)=\frac{1}{2}u^2$, a stationary solution to \eqref{JX2} satisfies
\begin{equation}\label{burstat}
\varepsilon a^2\partial_x u= \frac{u^2}{2}-\frac{C^2}{2}, \quad v=\frac{C^2}{2}, 
\end{equation}
with boundary conditions $u(\pm \ell)=\mp u^*$, for some $u^*>0$.
An approximate solution $U^\varepsilon(x;\xi)$ to the first equation of  \eqref{burstat} is obtained by matching two different steady states satisfying, respectively, the left and the right boundary conditions together with the request $U^\varepsilon|_{x=\xi}=0$ (see  \cite[Example 2.1]{MasStr12}). In formula
\begin{equation}\label{Ueps}
U^\varepsilon(x;\xi)= \left\{\begin{aligned}
&k_{-}\tanh{(k_{-}(\xi-x)/2\varepsilon)} \quad \rm in \ (-\ell,\xi), \\
&k_{+} \tanh{(k_{+}(\xi-x)/2\varepsilon)} \quad \rm in \ (\xi,\ell),
  \end{aligned}\right.
\end{equation}
where $a=1$, and $k_{\pm}$ are chosen so that the boundary conditions are satisfied
\begin{equation}\label{BCex1}
k_{\pm} \tanh {\left( \frac{k_{\pm}}{2\varepsilon}(\xi \mp \ell)\right)}=u_{\pm}.
\end{equation}
Moreover, by the condition $v=\frac{C^2}{2}$, we have
\begin{equation*}
V^\varepsilon(x;\xi)= \left\{\begin{aligned}
&k^2_- /2 \quad \rm in \ (-\ell,\xi), \\
&k^2_+/2 \quad \rm in \ (\xi,\ell).
  \end{aligned}\right.
\end{equation*}
}
\end{example}

\subsection{The linearized problem }

As already stated before, in order to describe the dynamics generated by an initial configuration localized far from the steady state $(\bar U^\varepsilon,\bar V^\varepsilon)$, we assume to have a one-parameter family 
$$\textbf{W}^{\varepsilon}(x;\xi(t)):=\{U^{\varepsilon}(x;\xi(t)), V^\varepsilon(x;\xi(t))\}_{\xi \in I},$$
parametrized by $\xi(t) \in I$, such that the couple $(U^{\varepsilon}(x;\xi(t)),V^{\varepsilon}(x;\xi(t)))$ is an approximate stationary solution to \eqref{JX2}, in the sense that it satisfies the stationary equation up to an error that is small in $\varepsilon$.
More precisely, following the idea firstly introduced in \cite{MasStr12}, we assume that there exist two families of smooth functions $\Omega_1^\varepsilon=\Omega_1^\varepsilon(\xi)$ and $\Omega_2^\varepsilon=\Omega_2^\varepsilon(\xi)$, uniformly convergent to zero as $\varepsilon \to 0$, such that, for any $\xi \in I$, the following estimates hold
  \begin{equation}\label{ipsufam1}
  \begin{aligned}
&|\langle \psi(\cdot), \mathcal P_1^\varepsilon[\textbf{W}^\varepsilon(\cdot,\xi)] \rangle | \leq |\Omega_1^\varepsilon(\xi)| |\psi|_{_{L^\infty}} \quad \forall \psi \in C(I), \\
&|\langle \psi(\cdot), \mathcal P_2^\varepsilon[\textbf{W}^\varepsilon(\cdot,\xi)] \rangle | \leq |\Omega_2^\varepsilon(\xi)| |\psi|_{_{L^\infty}} \quad \forall \psi \in C(I).
\end{aligned}
       \end{equation}  
Once a one-parameter family $\{ \textbf{W}^{\varepsilon}(\cdot;\xi) \}$ satisfying \eqref{ipsufam1} is chosen, we look for a solution to \eqref{JX2} in the form
        \begin{equation*}
\left \{ \begin{aligned}
u(\cdot, t) &= U^{\varepsilon}(\cdot;\xi(t))+u^1(\cdot,t), \\
v(\cdot,t) & = V^{\varepsilon}(\cdot;\xi(t))+v^1(\cdot,t).
\end{aligned} \right.
       \end{equation*}
Thus we describe the dynamics in a neighborhood of the family  $\{ U^{\varepsilon}(\cdot;\xi), V^{\varepsilon}(\cdot;\xi) \}$ using as coordinates the parameter $\xi$ and a distance vector  $Y=(u^1,v^1)$ , determined by the difference between the solution $(u,v)$ and an element of the approximate family. Substituting in \eqref{JX2}, we obtain
      \begin{equation*}
\left \{ \begin{aligned}
&\partial_t u^1+\partial_\xi  U^{\varepsilon}(\cdot; \xi)\frac{d\xi}{dt}+\partial_x V^{\varepsilon}(\cdot;\xi)+\partial_x v^1=0, \\
& \partial_t v^1+\partial_\xi V^\varepsilon(\cdot;\xi)\frac{d\xi}{dt}+a^2(\partial_x U^{\varepsilon}(\cdot; \xi)+\partial_x u^1)= \frac{1}{\varepsilon} \left\{f(U^{\varepsilon}(\cdot;\xi)+u^1)-V^{\varepsilon}(\cdot;\xi)-\!v^1\right\}.
\end{aligned} \right.
      \end{equation*}
Since $f(U^{\varepsilon}+u^1)=f(U^{\varepsilon})+f'(U^{\varepsilon})u^1+o(u^1)$ , we get
      \begin{equation}\label{1equv}
\left \{ \begin{aligned}
\partial_t u^1&=-\partial_x v^1-\partial_\xi U^{\varepsilon}(\cdot;\xi)\frac{d\xi}{dt}+\mathcal P_1^\varepsilon[\textbf{W}^{\varepsilon} (\cdot;\xi)] ,\\
\partial_t v^1&=-a^2\partial_x u^1+ \frac{1}{\varepsilon} (f'(U^{\varepsilon}(\cdot,\xi))u^1-v^1)-\partial_\xi V^{\varepsilon}(\cdot;\xi)\frac{d\xi}{dt}\\
 &\quad +\mathcal P_2^\varepsilon[\textbf{W}^{\varepsilon}(\cdot,\xi)]+\mathcal Q^\varepsilon[u^1],
\end{aligned} \right.
     \end{equation}
where
     \begin{equation*}
\left \{ \begin{aligned}
\mathcal P_1^\varepsilon[\textbf{W}^{\varepsilon}] &:=-\partial_x V^{\varepsilon}, \\
\mathcal P_2^\varepsilon[\textbf{W}^{\varepsilon}] &:=- a^2 \partial_xU^{\varepsilon}+\frac{1}{\varepsilon}(f(U^{\varepsilon})-V^{\varepsilon}), \\
\mathcal Q^\varepsilon[u]&:=o(u).
\end{aligned} \right.
      \end{equation*}

\begin{example}\label{ex2}
\rm{ Let us recall the Example \ref{ex1}, where we construct an approximate stationary solution for the Jin-Xin system  with $f(u)=u^2/2$ and $a=1$. We want to compute in this specific case
     \begin{equation*}
 \mathcal P_1^\varepsilon[\textbf{W}^\varepsilon(\cdot;\xi)]\!:=\!-\partial_x V^\varepsilon(\cdot;\xi), \quad\!\mathcal P_2^\varepsilon[\textbf{W}^\varepsilon(\cdot;\xi)]\!:= \!-\partial_x U^\varepsilon(\cdot;\xi)+\frac{1}{\varepsilon}\left( (U^{\varepsilon}(\cdot;\xi))^2/2\!-\!V^\varepsilon(\cdot;\xi)\right).
       \end{equation*}
From the explicit formula for $U^\varepsilon(x;\xi)$ given in \eqref{Ueps}, we get $\mathcal P_2^\varepsilon[\textbf{W}^{\varepsilon}] \equiv0$. 

On the other hand,  $-\partial_x V^\varepsilon(x;\xi)=\varepsilon \partial_x^2U^\varepsilon(x;\xi)-\partial_x f(U^\varepsilon(x;\xi))$. By direct substitution, we obtain the identity
   \begin{equation*}
\mathcal P_1^\varepsilon[\textbf{W}^{\varepsilon}(\cdot,\xi)]=[\![\partial_xU^\varepsilon]\!]_{x=\xi}\delta_{x=\xi}
\end{equation*}
in the sense of distributions. We also have
    \begin{equation*}
[\![\partial_x U^\varepsilon]\!]_{x=\xi}=\frac{1}{2\varepsilon}(k_--k_+)(k_-+k_+).
\end{equation*}
In order to determine the behavior of $\mathcal P_1^\varepsilon[\textbf{W}^\varepsilon(\cdot;\xi)]$ for small $\varepsilon$, we need an asymptotic description of the values $k_{\pm}$. Following the idea of \cite{MasStr12}, let us set $k_{\pm}:=\mp u_{\pm}(1+h_{\pm})$ and $\Delta_{\pm}:=\ell \mp \xi$. Relation \eqref{BCex1} becomes
    \begin{equation*}
\tanh{\left( \mp \frac{u_{\pm}\Delta_{\pm}}{2 \varepsilon}(1+h_{\pm})\right)}=\frac{1}{1+h_{\pm}}.
\end{equation*}
Therefore, the values $h_{\pm}$ are both positive and then
    \begin{equation*}
\tanh{\left( \mp \frac{u_{\pm}\Delta_{\pm}}{2\varepsilon}\right)} \leq \frac{1}{1+h_{\pm}},
\end{equation*}
that gives the asymptotic representation
    \begin{equation}\label{hpm}
h_{\pm} \leq \frac{1}{\tanh{(\mp u_{\pm}\Delta_{\pm}/2\varepsilon)}}-1= \frac{2}{e^{\mp u_{\pm}\Delta_{\pm}/\varepsilon}-1}= 2e^{\pm u_{\pm}\Delta_{\pm}/\varepsilon}+ l.o.t.,
\end{equation}
where {\it  l.o.t.} denotes lower order terms. Finally
 \begin{equation*}
[\![\partial_x U^\varepsilon]\!]_{x=\xi}=\frac{1}{2\varepsilon}(k_--k_+)(k_-+k_+)=\frac{u^2_*}{\varepsilon}(h_--h_+)+ l.o.t.,
\end{equation*}
where $u_\pm =\mp u^*$ for some $u^*>0$, so that we end up with
    \begin{equation}\label{asyP1}
[\![\partial_x U^\varepsilon]\!]_{x=\xi}\leq\frac{u^2_*}{\varepsilon}(e^{-u_*(\ell+\xi)/\varepsilon}-e^{-u_*(\ell-\xi)/\varepsilon})+l.o.t.,
\end{equation}
showing that this term is exponentially small for $\varepsilon \to 0$} and it is null when $\xi=0$, that corresponds to the equilibrium location of the shock when $f(u)=u^2/2$.

 In this case, if we neglect the lower order terms, we can write an asymptotic formula for $\Omega_1^\varepsilon$, that is
\begin{equation}\label{omega1eps}
\Omega_1^\varepsilon(\xi) \sim \frac{u^2_*}{\varepsilon}(e^{-u_*(\ell+\xi)/\varepsilon}-e^{-u_*(\ell-\xi)/\varepsilon}).
\end{equation}
Also, from \eqref{asyP1}, it follows that the quantity $|k_--k_+|$ is exponentially
small as $\varepsilon\to 0$, uniformly in any compact subset of $(-\ell,\ell)$;
therefore, for any $\delta\in(0,\ell)$, there exist constants $c_1, c_2>0$, indipendent on $\varepsilon$,
such that
\begin{equation}\label{boundUx}
	\bigl|[\![ \partial_x U^{\varepsilon}]\!]_{{x=\xi}}\bigr|
	\leq c_1\,e^{-c_2/\varepsilon}
	\qquad\qquad \forall\,\xi\in(-\ell+\delta,\ell-\delta).
\end{equation}
In particular, hypothesis \eqref{ipsufam1} is satisfied in the special case of $f(u)=u^2/2$.
 
\vskip0.25cm
We can also numerically compute the limit of the solution $(U^\varepsilon,V^\varepsilon)$ for $\varepsilon \to 0^+$.
For fixed $\xi$, we observe that, as $\varepsilon$ becomes smaller, the transition between $u_-$ and $u_+$ becomes more sharp, while $v$ tends to $f(u^*) \delta_{x=\xi}$, according to the fact that, in the limit $\varepsilon \to 0^+$, the solution $(U^\varepsilon(\cdot;\xi),V^\varepsilon(\cdot;\xi))$ converges to $(U_{{}_{\rm hyp}}(\cdot;\xi), V_{{}_{\rm hyp}}(\cdot;\xi))$ (see Fig. \ref{fig2}).

\begin{figure}[ht]
\centering
\includegraphics[width=15cm,height=13.5cm]{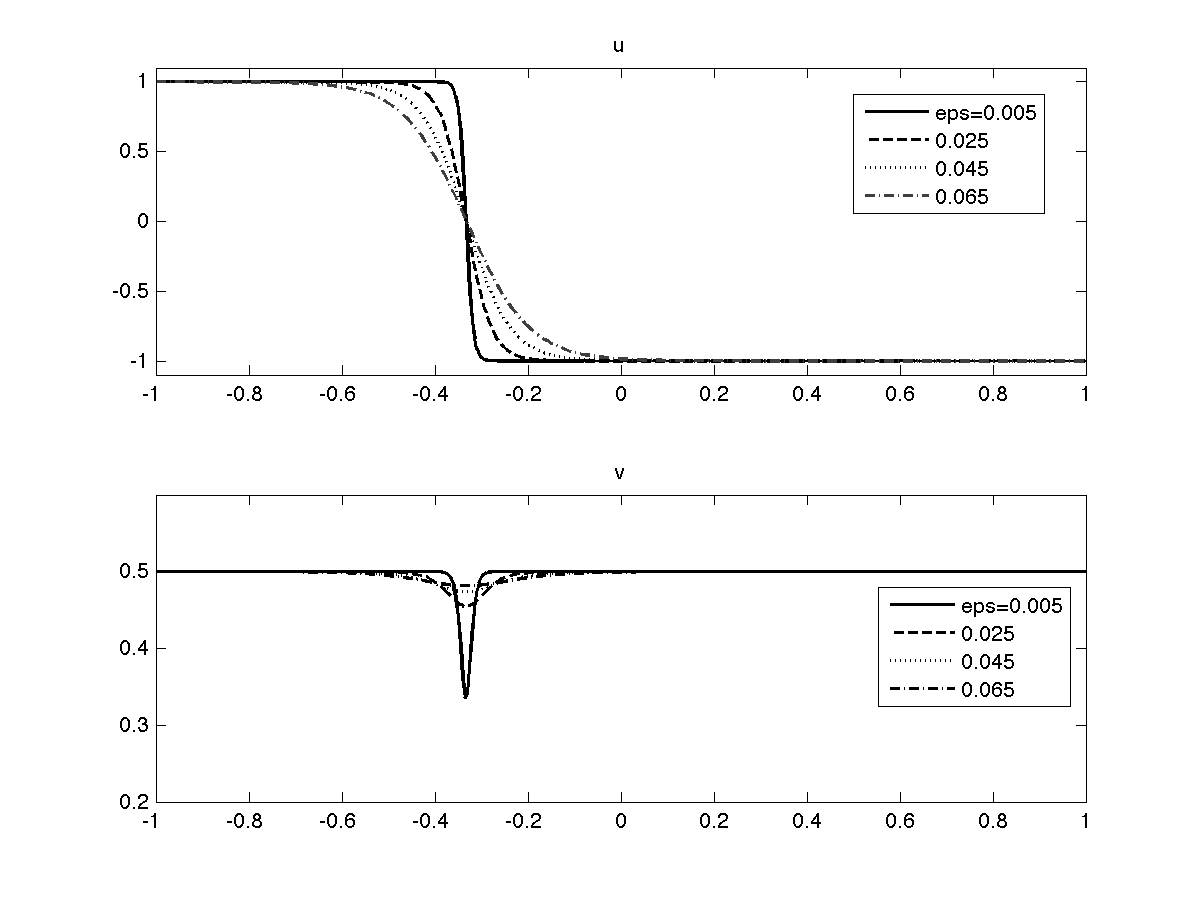}

\caption{\small{Profile of the stationary solution $(u,v)=(U^\varepsilon,V^\varepsilon)$  when $f(u)=u^2/2$. The steepening of the shock layer and the convergence to a Delta function of $v$ as $\varepsilon$ becames smaller are depicted.}}\label{fig2}
\end{figure}

\end{example}

\vskip0.25cm
Let us go back to the system \eqref{1equv}. From now on, in order to simplify the presentation, we set $(u^1,v^1)=(u,v)$ and
\begin{equation}\label{Deflinop}
Y=\left ( \begin{aligned}
u \\
v
\end{aligned} \right),
\qquad 
\mathcal L_\xi^\varepsilon Y:=\left ( \begin{aligned}
&-\partial_x v \\
-a^2\partial_x u +&\frac{1}{\varepsilon} (f'(U^\varepsilon)u-v)
\end{aligned} \right).
\end{equation}
Moreover, we introduce the following notation: if $\psi$, $\phi \in \C$, then $ \langle \psi,\phi \rangle := \int_I \bar \psi \, \phi$, while if $\boldsymbol \psi =(\psi_1,\psi_2)$ and $\boldsymbol \phi=(\phi_1,\phi_2)$, then $ \langle\boldsymbol{ \psi},\boldsymbol{\phi} \rangle := \langle \psi_1,\phi_1\rangle + \langle \psi_2,\phi_2 \rangle$.

\vskip0.2cm
Mimicking the approach of \cite{MasStr12} let us assume that, for any $\xi$, the linear operator $\mathcal{L}_\xi^\varepsilon$ has a sequence of eigenvalues $\lambda^\varepsilon_k=\lambda^\varepsilon_k(\xi)$ with corresponding (right) eigenfunctions $\boldsymbol{\phi}^\varepsilon_k=\boldsymbol{\phi}^\varepsilon_k(\xi,\cdot)$ (for more details see Section 3). Denoting by $\boldsymbol{\psi}^\varepsilon_k=\boldsymbol{\psi}^\varepsilon_k(\xi,\cdot)$ the eigenfunctions of the corresponding adjoint operator $\mathcal{L}_\xi^{\varepsilon^\ast}$ and setting $Y_k=Y_k(\xi;t):=\langle \boldsymbol{\psi}^\varepsilon_k(\cdot;\xi),Y(\cdot,t)\rangle$, we impose that the component $Y_1=\langle \boldsymbol{\psi}^\varepsilon_1(\cdot;\xi),Y(\cdot,t)\rangle$ is identically zero. Indeed, since we will prove that the first eigenvalue $\lambda_1^\varepsilon$ is small in the limit $\varepsilon \to 0$, we set an algebraic condition ensuring orthogonality between $\boldsymbol{\psi}^\varepsilon_1$ and $Y$, in order to remove the singular part of the operator $\mathcal L^\varepsilon_\xi$. Precisely, we impose the first component of the solution $Y_1$ to be zero., so that we solve the equation in a subspace where the operator doesn't vanish.

Thus, denoting by $Y_0$ the initial datum for the perturbation, we have
\begin{equation}\label{hip1}
\frac{d}{dt} \langle \boldsymbol{\psi}^\varepsilon_1(\cdot;\xi(t)), Y(\cdot,t) \rangle =0 \quad \textrm{and} \quad \langle \boldsymbol \psi^\varepsilon_1(\cdot; \xi_0),Y_0(\cdot)\rangle=0,
\end{equation}
so that
\begin{equation*}
\langle \boldsymbol{\psi}^\varepsilon_1(\cdot,\xi),\partial_t Y \rangle + \langle \partial_{\xi}\boldsymbol{\psi}^\varepsilon_1(\cdot,\xi) \frac{d \xi}{d t},Y \rangle =0.
\end{equation*}
Since $\boldsymbol{\psi}^\varepsilon_1=(\psi_1^u,\psi_1^v)$ is the first (left) eigenfunction, there holds $\mathcal L_\xi^{\varepsilon,*} \boldsymbol{\psi}^\varepsilon_1=\lambda_1\boldsymbol{\psi}^\varepsilon_1$, that is
\begin{equation*}
\left ( \begin{aligned}
a^2\partial_x\psi^v_{1} +&\frac{1}{\varepsilon}f'(U^\varepsilon(\cdot;\xi))\psi_1^v\\
\partial_x\psi^u_{1} &-\frac{1}{\varepsilon} \psi_1^v
\end{aligned} \right)=
\lambda_1 \left ( \begin{aligned}
\psi_1^u \\
\psi_1^v
\end{aligned} \right).
\end{equation*}
Hence, from \eqref{hip1} we get
\begin{equation*}
\langle \left ( \begin{aligned}
\partial_\xi\psi^u_1 \frac{d\xi}{dt} \\
\partial_\xi\psi^v_1 \frac{d\xi}{dt}
\end{aligned} \right)\!\!,\!\!
\left ( \begin{aligned}
u \\
v
\end{aligned} \right) \rangle 
+\langle \left ( \begin{aligned}
\psi^u_1 \\
\psi^v_1
\end{aligned} \right)\!\!,
\mathcal L_\xi^\varepsilon \left ( \begin{aligned}
u \\
v
\end{aligned} \right) 
+ \left ( \begin{aligned}
-\partial_\xi U^\varepsilon(\cdot,\xi)& \frac{d\xi}{dt}\!+\!\mathcal P_1^\varepsilon[\textbf{W}^\varepsilon(\cdot,\xi)] \\
-\partial_\xi V^\varepsilon(\cdot,\xi)\frac{d\xi}{dt}\!+\!&\mathcal Q^\varepsilon[u]+\mathcal P_2^\varepsilon[\textbf{W}^\varepsilon(\cdot,\xi)]
\end{aligned} \right) \rangle \!=\!0.
\end{equation*}
Since $\langle \boldsymbol{\psi}^\varepsilon_1, {\mathcal L}^\varepsilon_{\xi}Y \rangle= \lambda^\varepsilon_1\langle \boldsymbol{\psi}^\varepsilon_1, Y \rangle=0$, we have
\begin{equation*}
\begin{aligned}
&\langle \partial_\xi \psi_1^u(\cdot,\xi) \frac{d\xi}{dt},u \rangle +\langle \psi_1^u(\cdot;\xi),-\partial_\xi U^\varepsilon(\cdot,\xi) \frac{d\xi}{dt}+\mathcal P_1^\varepsilon[\textbf{W}^{\varepsilon}(\cdot;\xi)] \rangle  \\
\quad &+\langle \partial_\xi \psi_1^v(\cdot;\xi) \frac{d\xi}{dt},v \rangle + \langle \psi_1^v(\cdot;\xi),-\partial_\xi V^\varepsilon(\cdot,\xi)\frac{d\xi}{dt}+\mathcal Q^\varepsilon[u]+\mathcal P_2^\varepsilon[\textbf{W}^{\varepsilon}(\cdot,\xi)] \rangle =0,
\end{aligned}
\end{equation*}
and we end up with a scalar differential equation for the variable $\xi$, that is
\begin{equation}\label{1eqxi}
 \begin{aligned}
\frac{d\xi}{dt} = \frac{\langle \psi_1^v(\cdot;\xi),\mathcal Q^\varepsilon[u]+\mathcal P_2^\varepsilon[\textbf{W}^{\varepsilon}(\cdot,\xi)]\rangle+\langle \psi_1^u(\cdot,\xi), \mathcal P_1^\varepsilon[\textbf{W}^\varepsilon(\cdot;\xi)] \rangle}{\alpha^\varepsilon(\xi,u,v)} ,\\
\end{aligned} 
\end{equation}
where
\begin{equation*}
\alpha^\varepsilon(\xi,u,v) = - \langle \partial_\xi \psi_1^u (\cdot,\xi),u \rangle -\langle \partial_\xi \psi_1^v(\cdot;\xi) ,v \rangle+\langle \psi_1^u(\cdot;\xi),\partial_\xi U^\varepsilon \rangle+\langle \psi_1^v(\cdot;\xi),\partial_\xi V^\varepsilon \rangle.
\end{equation*}
Since we are interested in the regime $Y \sim 0$, the equation \eqref{1eqxi} is approximately solved for small $Y$. Thus the term $1/\alpha^\varepsilon(\xi,u,v)$ is expanded for $u,v \sim 0$, yielding
      \begin{equation*}
\begin{aligned}
\frac{1}{\alpha^\varepsilon(\xi,Y)}
&= \frac{1}{\langle \boldsymbol{\psi}^\varepsilon_1(\cdot;\xi), \partial_{\xi}\textbf{W}^{\varepsilon} \rangle}+ \frac{1}{\langle\boldsymbol{\psi}^\varepsilon_1(\cdot;\xi), \partial_{\xi}\textbf{W}^{\varepsilon} \rangle^2}\langle \partial_{\xi} \boldsymbol{\psi}^\varepsilon_1(\cdot;\xi),Y \rangle + R_1,\\
R_1 &=\frac{1}{\langle \boldsymbol{\psi}^\varepsilon_1(\cdot;\xi), \partial_{\xi}\textbf{W}^{\varepsilon} \rangle \!- \!\langle \partial_{\xi} \boldsymbol{\psi}^\varepsilon_1(\cdot;\xi),Y \rangle}\!-\!\frac{1}{\langle \boldsymbol{\psi}^\varepsilon_1(\cdot;\xi), \partial_{\xi}\textbf{W}^{\varepsilon} \rangle}\!-\! \frac{\langle \partial_{\xi} \boldsymbol{\psi}^\varepsilon_1(\cdot;\xi),Y \rangle}{\langle\boldsymbol{\psi}^\varepsilon_1(\cdot;\xi), \partial_{\xi}\textbf{W}^{\varepsilon} \rangle^2} \\
&= \frac{\langle \partial_{\xi}\boldsymbol{\psi}^\varepsilon_1(\cdot;\xi),Y\rangle^2}{\big[  \langle \boldsymbol{\psi}^\varepsilon_1(\cdot;\xi),\partial_{\xi}\textbf{W}^{\varepsilon} \rangle- \langle \partial_{\xi} \boldsymbol{\psi}^\varepsilon_1,(\cdot;\xi) Y \rangle  \big] \langle \boldsymbol{\psi}^\varepsilon_1(\cdot;\xi), \partial_{\xi}\textbf{W}^{\varepsilon} \rangle^2},
\end{aligned}
\end{equation*}
where 
\begin{equation*}
\langle \boldsymbol{\psi}^\varepsilon_1(\cdot;\xi), \partial_{\xi}\textbf{W}^{\varepsilon} \rangle:=\langle \psi^u_1(\cdot;\xi), \partial_{\xi}U^{\varepsilon} \rangle+\langle \psi^v_1(\cdot;\xi), \partial_{\xi}V^{\varepsilon} \rangle.
\end{equation*}
Now, for sake of simplicity, let us call $\alpha_0^\varepsilon(\xi):=\langle \boldsymbol{\psi}^\varepsilon_1(\cdot;\xi),\partial_\xi \textbf{W}^\varepsilon(\cdot,\xi) \rangle$.
Thus we end up with the nonlinear equation for $\xi(t)$, which reads
     \begin{equation}\label{eqxiNL}
\frac{d\xi}{dt}=\theta^\varepsilon(\xi)  \left( 1 +\frac{ \langle \partial_\xi \boldsymbol{\psi}^\varepsilon_1,Y\rangle}{\alpha_0^\varepsilon(\xi)} \right) +\rho^\varepsilon[\xi,Y], \qquad \langle \boldsymbol \psi^\varepsilon_1(\cdot; \xi_0), Y_0(\cdot) \rangle=0,
\end{equation}
where
       \begin{equation}\label{def}
 \left\{\begin{aligned}
 &\theta^\varepsilon(\xi):=\frac{ \langle \boldsymbol{\psi}^\varepsilon_1,{\mathcal F^\varepsilon[\textbf{W}^{\varepsilon}] \rangle}}{\alpha_0^\varepsilon(\xi)},\\
 &\rho^\varepsilon[\xi,Y]:=\theta_1(\xi,Y) \left ( 1 + \frac{\langle \partial_\xi \boldsymbol{\psi}^\varepsilon_1,Y\rangle}{\alpha_0^\varepsilon(\xi)} \right )+ \langle \boldsymbol{\psi}^\varepsilon_1, {\mathcal F}^\varepsilon[\textbf{W}^{\varepsilon}]+ { \boldsymbol{\mathcal Q}^\varepsilon[Y]} \rangle R_1,\\
 &R_1 := \frac{\langle \partial_\xi \boldsymbol{\psi}^\varepsilon_1(\cdot;\xi),Y\rangle^2}{[\alpha_0^\varepsilon(\xi)-\langle \partial_\xi \boldsymbol{\psi}^\varepsilon_1(\cdot;\xi),Y\rangle] (\alpha_0^\varepsilon(\xi))^2}, \\
&\theta_1(\xi,v): = \frac{\langle\boldsymbol{\psi}^\varepsilon_1, {\boldsymbol{\mathcal Q}^\varepsilon[Y]}\rangle}{\alpha_0^\varepsilon(\xi)} \\
&\boldsymbol{\mathcal Q}^\varepsilon[Y]= (0,\mathcal Q^\varepsilon[u]), \\
\end{aligned}\right. 
    \end{equation}
and $\mathcal F^\varepsilon$ is defined as in \eqref{compactform}. Equation \eqref{eqxiNL} has to be coupled with the equation for the perturbation $Y$. To this end, \eqref{1equv} is rewritten in the form
\begin{equation}\label{compatta}
\begin{aligned}
\partial_tY&=\mathcal L_\xi^\varepsilon Y -\partial_\xi \textbf{W}^\varepsilon(\cdot; \xi)  \frac{d  \xi}{dt}+ \mathcal F^\varepsilon[\textbf{W}^\varepsilon] + \boldsymbol{\mathcal Q}^\varepsilon[Y].\\
\end{aligned}
\end{equation}
Using \eqref{eqxiNL}, we end up with the following equation 
\begin{equation}\label{equvNL}
\partial_tY=(\mathcal L_\xi^\varepsilon +\mathcal M_\xi^\varepsilon)Y +H^\varepsilon(x;\xi)+\mathcal R^\varepsilon[Y,\xi],
\end{equation}
where
\begin{equation*}
\begin{aligned}
\mathcal M_\xi^\varepsilon Y &=\frac{1}{\alpha_0^\varepsilon(\xi)}\left ( \begin{aligned}
-\partial_\xi U^\varepsilon(\cdot;\xi) \,\theta^\varepsilon(&\xi)\,\langle \partial_\xi \boldsymbol{\psi}^\varepsilon_1(\cdot;\xi), Y \rangle\\
-\partial_\xi V^\varepsilon(\cdot;\xi) \,\theta^\varepsilon(&\xi)\, \langle \partial_\xi \boldsymbol{\psi}^\varepsilon_1(\cdot;\xi), Y \rangle \end{aligned}\right), \\
H^\varepsilon(x;\xi)&=\left( \begin{aligned}
\mathcal P_1^\varepsilon[ \textbf{W}^\varepsilon(\cdot;\xi)]&-\partial_\xi U^\varepsilon(\cdot,\xi) \theta^\varepsilon(\xi) \\
 \mathcal P_2^\varepsilon[\textbf{W}^\varepsilon(\cdot,\xi)]&-\partial_\xi V^\varepsilon(\cdot,\xi) \theta^\varepsilon(\xi) \end{aligned} \right), \\
\mathcal R^\varepsilon[Y,\xi]&=\left( \begin{aligned}
-\partial_\xi U^\varepsilon (&\cdot;\xi)\,\rho^\varepsilon[\xi,Y] \\
-\partial_\xi V^\varepsilon (\cdot;\xi)\,&\rho^\varepsilon[\xi,Y]+\mathcal Q^\varepsilon[u] \end{aligned} \right).
\end{aligned} 
\end{equation*}
Hence we obtain the following coupled system for the shock layer location $\xi(t)$ and the perturbation $Y$
 \begin{equation}\label{xiYcompleto}
\left\{\begin{aligned}
\frac{d\xi}{dt}&=\theta^\varepsilon(\xi)  \left( 1 + \frac{ \langle \partial_\xi \boldsymbol{\psi}^\varepsilon_1,Y\rangle}{\alpha_0^\varepsilon(\xi)} \right) +\rho^\varepsilon[\xi,Y], \\
\partial_t Y&=(\mathcal L_\xi^\varepsilon +\mathcal M_\xi^\varepsilon)Y +H^\varepsilon(x;\xi)+\mathcal R^\varepsilon[\xi,Y].
\end{aligned}\right.
   \end{equation}

  \begin{example}\label{ex3}\rm{
Let us consider the Jin-Xin system, for which one obtains
\begin{equation*}
\left\{\begin{aligned}
&\mathcal P_1^\varepsilon [\textbf{W}^\varepsilon]=-\partial_x V^\varepsilon(\cdot;\xi), \\
&\mathcal P_2^\varepsilon [\textbf{W}^\varepsilon]=-a^2\partial_{x}U^\varepsilon(\cdot;\xi)+\frac{1}{\varepsilon}(f(U^\varepsilon(\cdot;\xi))-V^\varepsilon(\cdot;\xi)). 
\end{aligned}\right.
\end{equation*}
For what concerns the linear operator, setting $b^\varepsilon(x;\xi):=f'(U^\varepsilon(\cdot,\xi)) $, and recalling the definition of $\mathcal L^\varepsilon_\xi$ given in \eqref{Deflinop}, we get the following expression for the adjoint operator $\mathcal L_\xi^{\varepsilon,*}$
\begin{equation*}
%\mathcal L_\xi^\varepsilon Y:=\left ( \begin{aligned}
%&-\partial_x v \\
%-a^2\partial_x u+&\frac{1}{\varepsilon}(b^\varepsilon(\cdot;\xi)u-v)
%\end{aligned} \right),
%\qquad 
\mathcal L_\xi^{\varepsilon,*} Y:=\left ( \begin{aligned}
a^2 \partial_x v\, +\, \frac{1}{\varepsilon}&b^\varepsilon(\cdot;\xi)v \\
 \partial_x u\, - \,&\frac{1}{\varepsilon} v
\end{aligned} \right),
\end{equation*}
complemented with Dirichlet boundary conditions. 
To obtain an asymptotic expression for the function $\theta^\varepsilon(\xi)$, we  need to approximately compute the functions $\boldsymbol\psi^\varepsilon_1$ and $\partial_\xi \textbf{W}^\varepsilon$. As usual, we refer to the case $f(u)=u^2/2$.

For $\varepsilon \sim 0$, the function $\boldsymbol\psi^\varepsilon_1$ is close to the eigenfunction $\boldsymbol\psi^0_1=(\psi_1^{0,u},\psi_1^{0,v})$ of the operator $\mathcal L^{0,*}_\xi$ relative to the eigenvalue $\lambda=0$, with
\begin{equation*}
b^0(x;\xi):= u_- \chi_{(-\ell,\xi)}(x)+u_+ \chi_{(\xi,\ell)}(x).
\end{equation*}
For example, in $(-\ell,\xi)$ we have
\begin{equation*}
\left\{ \begin{aligned} 
&a^2\partial_x \psi_1^{0,v}+\frac{u_-}{\varepsilon}\psi_1^{0,v}=0, \\
&\partial_x \psi_1^{0,u}-\frac{1}{\varepsilon}\psi_1^{0,v}=0, \\
& {\psi}_1^{0,u}(-\ell)=0, \quad [\![ {\psi}_1^{0,u} ]\!]_{\xi}=0,
  \end{aligned}\right.
\end{equation*}
that is  $\psi_1^{0,u}=A(1-e^{-\frac{u_-}{a^2\varepsilon}(x+x_0)})$ and $\psi_1^{0,v}=\varepsilon \partial_x\psi_1^{0,u}$, where $x_0$ is an integration constant. By imposing the conditions on the boundary and on the jump, and by doing the same computations in the interval $(\xi,\ell)$, we obtain
      \begin{equation*}
\psi^u_1(x)\sim \psi^{0,u}_1(x)=\left\{ \begin{aligned} 
&(1-e^{u_+(\ell-\xi)/a^2\varepsilon})(1-e^{-u_-(\ell+x)/a^2\varepsilon}) \quad x<\xi, \\
&(1-e^{-u_-(\ell+\xi)/a^2\varepsilon})(1-e^{u_+(\ell-x)/a^2\varepsilon})  \quad x > \xi,
  \end{aligned}\right.
\end{equation*}
    \begin{equation*}
\psi^v_1(x)\sim \psi^{0,v}_1(x)=\left\{ \begin{aligned} 
&\frac{u_-}{a^2}(1-e^{u_+(\ell-\xi)/a^2\varepsilon})e^{-u_-(\ell+x)/a^2\varepsilon} \quad x<\xi, \\
&-\frac{u_+}{a^2}(1-e^{-u_-(\ell+\xi)/a^2\varepsilon})e^{e^{u_+(\ell-x)/a^2\varepsilon}}  \quad x > \xi,
  \end{aligned}\right.
\end{equation*}
so that $\boldsymbol\psi^\varepsilon_1 =(\psi_1^u,\psi_1^v) \sim (1,0)$ for $\varepsilon \sim 0$. Furthermore, with the approximation $U^\varepsilon(x;\xi)\sim U_{{}_{\rm hyp}}(x;\xi)$ and $V^\varepsilon(x;\xi) \sim V_{{}_{\rm hyp}}(x)$, we have
     \begin{equation*}
\begin{aligned}
&\frac{U^\varepsilon(x;\xi+h)-U^\varepsilon(x;\xi)}{h} \sim -\frac{1}{h}[\![u]\!]\chi_{_{(\xi,\xi+h)}}(x),\\
&\frac{V^\varepsilon(x;\xi+h)-V^\varepsilon(x;\xi)}{h} \sim -\frac{1}{h}[\![f(u)]\!]\chi_{_{(\xi,\xi+h)}}(x),
\end{aligned}
\end{equation*}
so that $\partial_\xi U^\varepsilon$ and $\partial_\xi V^\varepsilon$ converge to $-[\![u]\!]\delta_\xi$ and $-[\![f(u)]\!]\delta_{\xi}$ respectively as $\varepsilon \to 0$ in the sense of distributions. Thus, since $\langle \boldsymbol\psi^\varepsilon_1,\partial_\xi \textbf{W}^\varepsilon \rangle \sim -[\![u]\!]$, we deduce an asymptotic expression for the function $\theta^\varepsilon$
\begin{equation*}
\theta^\varepsilon(\xi)\sim -\frac{1}{[\![u]\!]}\langle 1, \mathcal P_1^\varepsilon[\textbf{W}^\varepsilon] \rangle.
\end{equation*}
With the choice of $\textbf{W}^\varepsilon=(U^\varepsilon,V^\varepsilon)$ proposed in Example \ref{ex1}, such expression becomes
\begin{equation}\label{stimasutheta}
\theta^\varepsilon(\xi) \sim \frac{u^*}{\varepsilon}(e^{-u_*(l+\xi)/\varepsilon}-e^{-u_*(l-\xi)/\varepsilon}).
\end{equation} }

\end{example}

\section{Spectral analysis}\label{3} 

In this section we analyze the spectrum of the linearized operator $\mathcal L_\xi^\varepsilon$ in order to determine a precise description of the location of the eigenvalues.

We recall that $\mathcal L^\varepsilon_\xi$ has been defined in \eqref{Deflinop},
%\begin{equation*}
%\mathcal L_\xi^\varepsilon Y:=\left ( \begin{aligned}
%&-\partial_x v \\
%-a^2\partial_x u +&\frac{1}{\varepsilon} (f'(U^\varepsilon)u-v)
%\end{aligned}\right)
%\end{equation*}
so that the eigenvalue problem $\mathcal L_\xi^\varepsilon \Phi =\lambda \Phi$ reads
\begin{equation*}
\left \{ \begin{aligned}
\lambda \varphi &=-\partial_x \psi,\\
\lambda \psi &= -a^2 \partial_x \varphi+ \frac{1}{\varepsilon} (f'(U^{\varepsilon})\varphi-\psi),
\end{aligned} \right.
\end{equation*}
complemented with Dirichlet boundary conditions. Hence, by differentiating the second equation with respect to $x$, we obtain
\begin{equation}\label{equivaut}
\varepsilon a^2 \partial_x^2 \varphi-\partial_x(f'(U^\varepsilon)\varphi)=\lambda(1+\varepsilon \lambda)\varphi. 
\end{equation}
Then we are interested in studying the eigenvalue problem for the linear differential diffusion-transport operator 
\begin{equation}\label{ldto}
\mathcal L^{\varepsilon,vsc} \varphi := \varepsilon a^2 \partial_x^2 \varphi-\partial_x (b^\varepsilon \varphi), \quad b^\varepsilon(x;\xi):=f'(U^\varepsilon(x;\xi)).
\end{equation}
In \cite{MasStr12} it is proven that, under appropriate hypotheses on the behavior of the function $b^\varepsilon(x;\xi)$ in the limit $\varepsilon \to 0$, the eigenvalues of $\mathcal L^{\varepsilon,vsc}$ have the following distribution
\begin{equation*}
-Ce^{-c/\varepsilon} \leq \lambda_1^{vsc} <0 \quad {\rm and} \quad \lambda_k^{vsc} \leq -\frac{C}{\varepsilon} \quad \forall \, k \geq 2.
\end{equation*}
More precisely, the following Propositions  are proven in \cite{MasStr12} (for more details, see \cite[Proposition 4.1 and Proposition 4.3]{MasStr12}).

\begin{proposition}\label{citata1}
Let $b^\varepsilon$ be a family of functions satisfying the assumption:

\quad \textbf {A0}. There exists a constant $C_0>0$, independent on $\varepsilon >0$, such that
\begin{equation*}
|b^\varepsilon|_{_{L^\infty}}+ \varepsilon \left| \frac{db^\varepsilon}{dx}\right|_{_{L^\infty}} \leq C_0.
\end{equation*} 
If there exist $\xi \in (-\ell,\ell)$, $b_+ < 0 < b_-$ and a constant $C_1>0$ for which $|b^\varepsilon-b^0|_{_{L^1}} \leq C_1\varepsilon$, where $b^0$ is the step function jumping from $b^-$ to $b^+$, then there exist constants $C,c>0$ such that $-Ce^{-c/\varepsilon} \leq \lambda_1^{vsc} < 0$.

\end{proposition}

\begin{proposition}\label{citata2}
Let $b^\varepsilon$ be a family of functions satisfying the assumptions:

\quad \textbf {A1}. $b^\varepsilon \in C^0[-\ell,\ell]$, $b^\varepsilon$ is twice  differentiable at any $x \neq \xi$ and
\begin{equation*}
\begin{aligned}
&\frac{db^\varepsilon}{dx}<0 <b^\varepsilon  \qquad {\rm and} \qquad  \frac{d^2b^\varepsilon}{dx^2} <0 <b^\varepsilon  \quad &{\rm in} \, &(-\xi,\ell), \\
&b^\varepsilon <0<\frac{d^2 b^\varepsilon}{dx^2} 
 \qquad {\rm and} \qquad  \frac{db^\varepsilon}{dx} <0<\frac{d^2 b^\varepsilon}{dx^2} \quad &{\rm in} \, &(\xi,\ell).
 \end{aligned}
\end{equation*}

\quad \textbf {A2}. For any $C_0>0$ there exists $c_0 >0$ such that, for any $x$ satisfying $|x-\xi| \geq c_0 \varepsilon$, there holds
\begin{equation*}
|b^\varepsilon-b^0| \leq C_0\varepsilon \quad {\rm and} \quad \varepsilon \left| \frac{db^\varepsilon}{dx} \right| \leq C_0.
\end{equation*}

\quad \textbf {A3}. The left (reps. the right) first order derivatives of $b^\varepsilon$ at $\xi$ exist and 
\begin{equation*}
 \liminf_{\varepsilon \to 0^+} \varepsilon \left| \frac{db^\varepsilon}{dx}(\xi \pm) \right| >0.
\end{equation*}
Then there exists a constant $C>0$ such that, for all $k \geq 2$, $\lambda_k^{vsc} \leq -C/\varepsilon$ for all $\varepsilon$ sufficiently small.

\end{proposition}

\begin{remark}\rm{
When $f(u)=u^2/2$, then $b^\varepsilon(x;\xi)=U^\varepsilon(x;\xi)$. With the choice of $U^\varepsilon$ proposed in Example \ref{ex1}, we can easily check that hypotheses  \textbf{A0-1-2-3} are verified.
}
\end{remark}
\vskip0.25cm

From \eqref{equivaut}, we observe that $\lambda$ is an eigenvalue of $\mathcal L^\varepsilon_\xi$ if and only if $\lambda^{vsc}:=\lambda(1+\varepsilon \lambda)$ is an eigenvalue for the operator $\mathcal L^{\varepsilon,vsc}$ defined in \eqref{ldto}.
Hence, if $\lambda=\lambda_n^{JX}$ is an eigenvalue of $\mathcal L^\varepsilon_\xi$, then there exists an eigenvalue $\lambda_n^{vsc}$ such that
\begin{equation*}
\varepsilon {\lambda_n^{JX}}^2+\lambda_n^{JX}=\lambda_n^{vsc},
\end{equation*}
so that
\begin{equation}\label{uguaauto}
\lambda_{n,\pm}^{JX}=-\frac{1}{2\varepsilon} \pm \frac{1}{2\varepsilon} \sqrt{1+4\varepsilon\lambda_n^{vsc}}.
\end{equation}
Hence, if $\lambda_n^{vsc} > -\frac{1}{4\varepsilon}$, then $\lambda_{n,\pm}^{JX} \in \R$. 
Moreover, since $\lambda_n^{vsc}$ are negative for all $n \in \N$
\begin{equation}\label{estlambda1}
\lambda_{n,+}^{JX}=\frac{2 \lambda_n^{vsc}}{1+\sqrt{1+4\varepsilon\lambda_n^{vsc}}} <0, \qquad \lambda_{n,-}^{JX}=\frac{-2\lambda_n^{vsc}}{\sqrt{1+4\varepsilon\lambda_n^{vsc}}-1}<0.
\end{equation}
Due to Propositions \ref{citata1} and \ref{citata2}, we know that $\lambda_1^{vsc}>-\frac{1}{4\varepsilon}$ and $\lambda_1^{vsc} \sim -e^{-C/\varepsilon}$ as $\varepsilon \to 0$. Thus, from \eqref{uguaauto} and \eqref{estlambda1}, there exists a constant $C'$ such that
\begin{equation*}
-e^{-C'/\varepsilon} \leq \lambda_{1,+}^{JX}<0, \qquad \lambda_{1,-}^{JX} \leq -\frac{1}{2\varepsilon}.
\end{equation*}
Moreover, if for some $n >1$ there exist other eigenvalues $\lambda_n^{vsc}$ such that $\lambda_n^{vsc} >-\frac{1}{4\varepsilon}$, then they are of order $1/\varepsilon$, so that
\begin{equation*}
\lambda_{n,\pm}^{JX} \leq -C''/\varepsilon.
\end{equation*}
On the other hand, if $\lambda_n^{vsc} < -\frac{1}{4\varepsilon}$, then $\lambda_{n,\pm}^{JX} \in \C$. More precisely
\begin{equation*}
\lambda_{n,\pm}^{JX}=-\frac{1}{2\varepsilon} \pm \frac{i}{2\varepsilon} \sqrt{|1+4\varepsilon\lambda_n^{vsc}|}.
\end{equation*}
Proposition \ref{citata2} assures that there exists $j \geq 2$ such that $\lambda_n^{vsc} < -\frac{1}{4\varepsilon}$ for all $n \geq j$, so that $Re(\lambda_{n,\pm}^{JX})$ and  $Im(\lambda_{n,\pm}^{JX})$ are terms of order $1/\varepsilon$.
For example, if $j=2$ and we take into account $\lambda_2^{vsc}<0$, the corresponding eigenvalues for $\mathcal L^\varepsilon_\xi$ verifies
\begin{equation*}
\begin{aligned}
Re(\lambda_{2,\pm}^{JX})=-\frac{1}{2\varepsilon}, \qquad
Im(\lambda_{2,\pm}^{JX})= \pm \frac{1}{2\varepsilon} \sqrt{|1+4\varepsilon\lambda_2^{vsc}|}.
\end{aligned}
\end{equation*}
Moreover, for $\lambda_{3,\pm}^{JX}$, since $|\lambda_3^{vsc}| >|\lambda_2^{vsc}| $, we have
\begin{equation*}
\begin{aligned}
Re(\lambda_{3,\pm}^{JX})=-\frac{1}{2\varepsilon}=Re(\lambda_{2,\pm}^{JX}), \qquad
|Im(\lambda_{3,\pm}^{JX})|=  \frac{1}{2\varepsilon} \sqrt{|1+4\varepsilon\lambda_3^{vsc}|} > |Im(\lambda_{2,\pm}^{JX})|.
\end{aligned}
\end{equation*}
Figure \ref{fig3} shows the connection between the two spectra when $j=2$, so that only the first two eigenvalues of $\mathcal L^\varepsilon_\xi$ belong to $\R$.

 \begin{figure}[ht]
\centering
\includegraphics[width=1\linewidth]{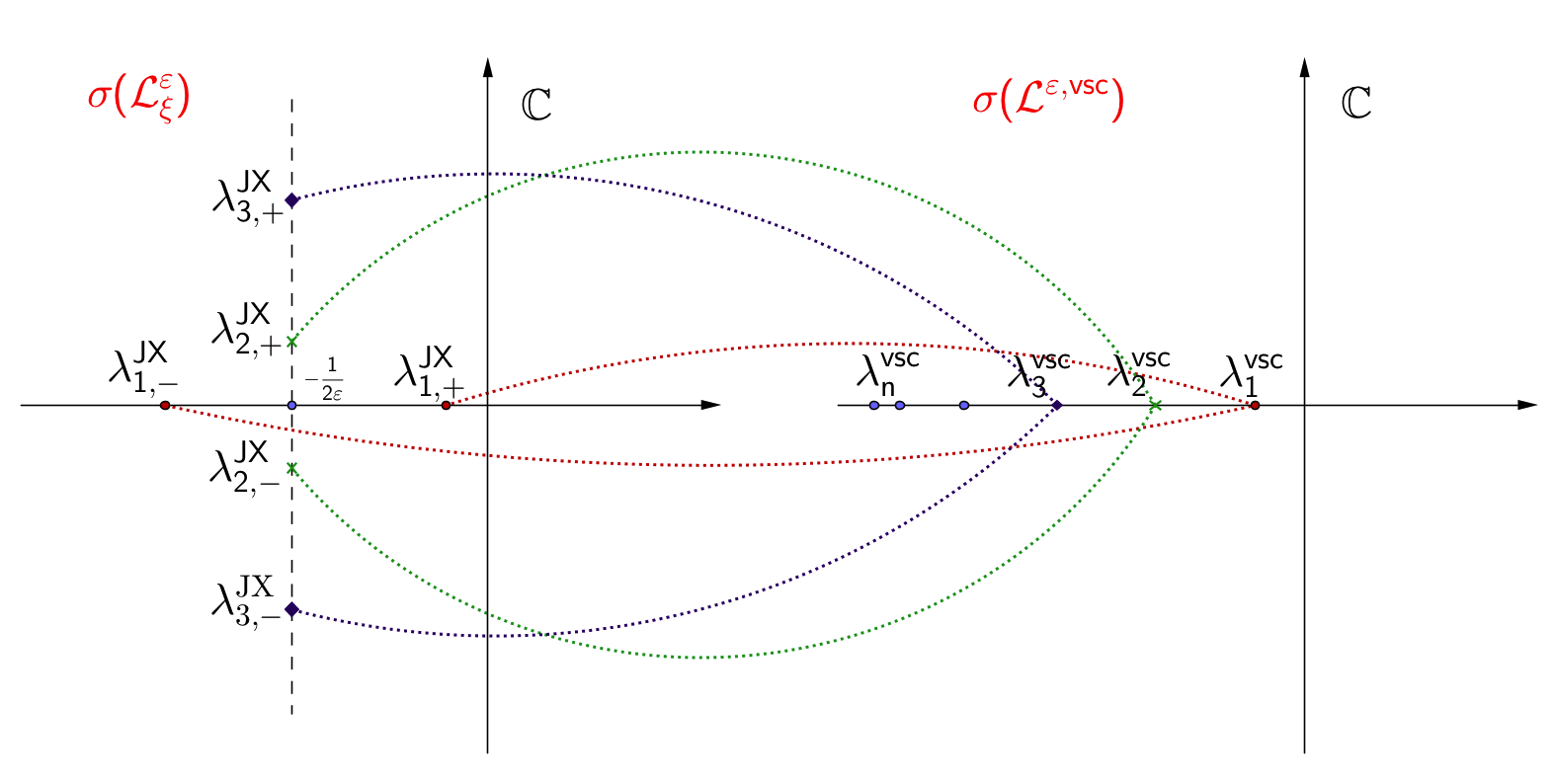}
\caption{\small{The spectra of the operators $ \mathcal L^\varepsilon_\xi$ and $\mathcal L^{\varepsilon,vsc}$.}}\label{fig3}
 \end{figure}
 
Hence, the following proposition holds
\begin{proposition}\label{spctrum}
Let $b^\varepsilon$ be a family of functions satisfying assumptions \textbf{A0-1-2-3} for some $\xi \in (-\ell, \ell)$ and for some $b_+ <0 <b_-$. Then the spectrum of the linearized operator $\mathcal L^\varepsilon_\xi$ can be decomposed as follow
\vskip0.3cm
 \textbf {i}. $\lambda_{1,+}^{JX} \in \R$ \ and  $ \quad -e^{-C'/\varepsilon}\leq \lambda^{JX}_{1,+} <0$.
 \vskip0.3cm
  \textbf {ii}. $\lambda_{1,-}^{JX} \in \R$ \ and $\quad   \lambda^{JX}_{1,-} \leq -1/\varepsilon$.
  \vskip0.3cm
\textbf{iii}. There exists $k \geq 0$ such that
$$\lambda_{n,\pm}^{JX} \in \R \quad {\rm and } \quad \lambda_{n,\pm}^{JX} \leq -C''/\varepsilon, \quad \forall \ n=2,...,1+k.$$
\vskip0.3cm
\textbf{iv}. $\lambda_{n,\pm}^{JX} \in \C$ for all $n \geq 2+k$ and
$$Re(\lambda_{n,\pm}^{JX})=-\frac{1}{2\varepsilon}, \quad Im(\lambda_{n,\pm}^{JX}) \sim \pm \frac{C}{\varepsilon.}$$
\end{proposition}

\begin{remark}\rm{ The case $k=0$ correspond to the case where the eigenvalues  $\lambda_{n,\pm}^{JX}$ are complex for all $n \geq 2$ (see also Figure \ref{fig3}). Indeed, in this case, Proposition \ref{spctrum},  step \textbf{iii}, assures that $\lambda_{n,\pm}^{JX} \in \R$ for all $n=2,...,1$, meaning that there no exists $n \in \N$, $n >1$ such that $\lambda_{n,\pm}^{JX} \in \R$. Furthermore, step {\bf iv}  states that $\lambda_{n,\pm}^{JX} \in \C$ for all $n \geq 2$.

}

\end{remark}
\begin{remark}\rm{ In \cite{KreissKreiss86}, Kreiss G. and Kreiss H. performed the spectral analysis for the operator
\begin{equation*}
\mathcal L_\varepsilon u := \varepsilon \partial_x^2u - \partial_x (f'(\bar U^\varepsilon(x))u),
\end{equation*}
arising from the linearization around the exact steady state $\bar U^\varepsilon(x)$ of
\begin{equation*}
 \partial_t u=\varepsilon\partial_x^2u-\partial_x f(u)
\end{equation*}
proving that all the eigenvalues are real and negative. By using this result and our spectral analysis, if we linearize the system \eqref{JX2} around the exact stationary solution $(\bar U^\varepsilon, \bar V^\varepsilon)$,  we can prove that the real part of all the eigenvalues of the linearized operator is negative, so that the steady state $(\bar U^\varepsilon, \bar V^\varepsilon)$ is asymptotically stable with exponential rate.

}
\end{remark}

\section{Asymptotic estimates for the first eigenvalue}\label{4}
 
In this section we want to study the behavior in $\varepsilon$ of the principal eigenvalue of the operator $\mathcal L_\xi ^\varepsilon$ associated to the linearization of \eqref{JX2} around an approximate stationary solution. Since usually the metastable behavior is the result of the presence of a first small eigenvalue, our aim is to determine an asymptotic expression for  $\lambda_{1,+}^{JX}$. We have already emphasized the fact that $\lambda^{JX}$ is an eigenvalue of the nonlinear Jin-Xin system if and only if $\lambda^{vsc}=\lambda^{JX}(1+\varepsilon \lambda^{JX})$ is an eigenvalue for the  operator $\mathcal L^{\varepsilon,vsc}$ defined in \eqref{ldto} where $a \equiv 1$ and
%\begin{equation}\label{opasy}
%\mathcal L^{\varepsilon,vsc} := \varepsilon  \partial_x^2 u - \partial_{x} (f'(U^\varepsilon(x;\xi))u),
%\end{equation}
where $U^\varepsilon(x;\xi(t))$ is an approximate stationary solution for the scalar conservation law
\begin{equation}\label{burger}
\left\{\begin{aligned}
&\partial_t u= \varepsilon \partial_x^2 u-\partial_x f(u), \\
&u(\pm \ell, t)=\mp u^*, \quad u(x,0)=u_0(x).
\end{aligned}\right.
\end{equation}
In particular
\begin{equation}\label{lambda1JX}
|\lambda_{1,+}^{JX}|= \frac{2|\lambda_1^{vsc}|}{1+\sqrt{1+4\varepsilon\lambda_1^{vsc}}}.
\end{equation}
In \cite{MasStr12}, in the special case $f(u)=u^2/2$, $U^\varepsilon(x;\xi(t))$ is given by \eqref{Ueps}. In that paper it is proven that, for $\varepsilon \sim 0$
\begin{equation*}
\lambda_1^{vsc}(\xi) \sim -\frac{{u^*}^2}{2\varepsilon}\left[  e^{-u^*\varepsilon^{-1}(\ell-\xi)}+e^{-u^* \varepsilon^{-1}(\ell+\xi)}\right] ,
\end{equation*}
so that
\begin{equation}\label{lambdaasybur}
|\lambda_{1,+}^{JX}(\xi)| \sim \frac{ \frac{{u^*}^2}{\varepsilon}\left[  e^{-u^*\varepsilon^{-1}(\ell-\xi)}+e^{-u^* \varepsilon^{-1}(\ell+\xi)}\right]}{1+ \sqrt{1-2  {u^*}^2 \left[  e^{-u^*\varepsilon^{-1}(\ell-\xi)}+e^{-u^*\varepsilon^{-1}(\ell+\xi)}\right]}}.
\end{equation}
This formula shows that the principal eigenvalue of the Jin-Xin system with $f(u)=u^2/2$ is exponentially small in $\varepsilon$.

\vskip0.25cm
In order to determine an asymptotic expression of the first eigenvalue of the operator \eqref{opasy} for a general class of flux function $f(u)$, we refer to the paper of Reyna L.G. and Ward M.J., \cite{ReynWard95}; here the authors use the method of matched asymptotic expansions (MMAE) to determine an approximate stationary solution to \eqref{burger}. 

Mimicking their approach and performing the same calculations as in \cite{ReynWard95}, with the appropriate changes due to the fact that the study of our equation is made in the interval $(-\ell,\ell)$ instead of $(0,1)$, we obtain that the leading order MMAE solution for $\varepsilon \to 0^+$ is given by a function $u_s(x;\xi) \sim u_s[\varepsilon^{-1}(x-\xi)]$, where $\xi \in (-\ell,\ell)$ and the shock profile $u_s(z)$ satisfies 
\begin{equation*}
\left\{\begin{aligned}
&u'_s(z)=f(u_s(z))-f(u^*), \quad -\infty < z < \infty,\\
&u_s(z) \sim u^*-z_- e^{\nu_- z}, \quad z \to -\infty, \\ 
&u_s(z) \sim -u^*+z_+ e^{-\nu_+ z}, \quad z \to +\infty.
\end{aligned}\right.
\end{equation*}
The positive constant $\nu_{\pm}$ and $z_{\pm}$ describe the tail behavior of $u_s(z)$ and are defined by
\begin{equation*}
\begin{aligned}
\nu_{\pm} &= \mp f'(\mp u^*), \\
\log \left( \frac{z_{\pm}}{u^*}\right) &= \pm \nu_{\pm} \int_0^{\mp u^*} \left[ \frac{1}{f(\eta)-f(u^*)} \pm \frac{1}{\nu_{\pm}(\eta \pm u^*)}\right] d\eta.
\end{aligned}
\end{equation*}
In particular, when $f(u)=u^2/2$, $u_s(z)=-u^* \tanh(u^* z/2)$, according to \eqref{Ustatbur}. Notice that the MMAE solution satisfies exactly the equation, while the boundary conditions are satisfy within exponentially small terms. Instead, the construction presented  in this paper in Example \ref{ex1} gives a function $U^\varepsilon(x;\xi)$ that verifies exactly the boundary conditions and solves approximately the stationary equation. 

The eigenvalue problem associated to the linearization around $u_s$ is given by
 \begin{equation}\label{eigenvsc}
\left\{\begin{aligned}
& L \phi \equiv \varepsilon^2 \partial_x^2 \phi - V[\varepsilon^{-1}(x -\xi)] \phi = \lambda \phi, \\
&\phi(\pm \ell)=0, \\ 
&V(z) = \frac{1}{4}[f'(u_s(z))]^2+\frac{1}{2} f''(u_s(z))u'_s(z).
\end{aligned}\right.
\end{equation}
In \cite{ReynWard95} it is proven that the first eigenvalue of \eqref{eigenvsc} has the following asymptotic representation (for details see \cite[Formula (2.14)]{ReynWard95})
\begin{equation*}
\lambda_1^{vsc}(\xi) \sim -\frac{1}{2u^*} \left[ a_+ \nu_+^2 e^{-\nu_+\varepsilon^{-1}(\ell-\xi)}+a_-\nu_-^2e^{-\nu_- \varepsilon^{-1}(\ell+\xi)}\right].
\end{equation*}
Finally, from \eqref{lambda1JX}, we get
\begin{equation}\label{lambdaasy}
|\lambda_{1,+}^{JX}(\xi)| \sim \frac{ \frac{1}{u^*}\left[ a_+ \nu_+^2 e^{-\nu_+\varepsilon^{-1}(\ell-\xi)}+a_-\nu_-^2e^{-\nu_- \varepsilon^{-1}(\ell+\xi)}\right]}{1+ \sqrt{1-\frac{2 \varepsilon}{u^*} \left[ a_+ \nu_+^2 e^{-\nu_+\varepsilon^{-1}(\ell-\xi)}+a_-\nu_-^2e^{-\nu_- \varepsilon^{-1}(\ell+\xi)}\right]}}.
\end{equation}
This formula shows that $\lambda_{1,+}^{JX}$ is exponentially small as $\varepsilon \to 0$. We remark that, when $f(u)=u^2/2$,  $a_+=a_-=2u^*$ and $\nu_+=\nu_-=u^*$, so that \eqref{lambdaasy} is the same as \eqref{lambdaasybur}.

\section{The behavior of the shock layer position}\label{5}

Let us consider the system \eqref{xiYcompleto} for the couple $(\xi,Y)$ and let us neglect the $o(Y)$ terms
\begin{equation}\label{xiYlineare1}
\left\{\begin{aligned}
\frac{d\xi}{dt}&=\theta^\varepsilon(\xi)  \left( 1 + \frac{ \langle \partial_\xi \boldsymbol{\psi}^\varepsilon_1,Y\rangle }{\alpha_0^\varepsilon(\xi)}\right), \\
Y_t&=(\mathcal L_\xi^\varepsilon +\mathcal M_\xi^\varepsilon)Y +H^\varepsilon(x;\xi).
\end{aligned}\right.
\end{equation}
This system is obtained by linearizing with respect to $Y$ and by keeping the nonlinear dependence on $\xi$, in order to describe the slow motion of the shock layer position far from the equilibrium location $\bar \xi$.

We complement the so called \textbf{quasi-linearized system} \eqref{xiYlineare1} with initial data\begin{equation*}
\xi(0)=\xi_0 \in (-\ell,\ell) \quad {\rm and} \quad Y(x,0)=(u_0(x),v_0(x)), \quad u_0,v_0 \in L^2(-\ell,\ell).
\end{equation*}
The aim of this section is to analyze the behavior  of the solution to \eqref{xiYlineare1} in the limit of small $\varepsilon$. Subsequently, we will prove a result that characterizes the behavior of the shock layer location, proving that it moves towards the unique stationary solution with exponentially small rate.

Before stating our  result, let us recall the assumptions. 

{\bf H1.} Let the family $\{\textbf{W}^\varepsilon(\cdot;\xi) \}$ be such that there exist two families of smooth functions $\Omega_1^\varepsilon$  and $\Omega_2^\varepsilon$ such that 
\begin{equation*}
  \begin{aligned}
&|\langle \psi(\cdot), \mathcal P_1^\varepsilon[\textbf{W}^\varepsilon(\cdot,\xi)] \rangle | \leq |\Omega_1^\varepsilon(\xi)| |\psi|_{_{L^\infty}} \quad \forall \psi \in C(I), \\
&|\langle \psi(\cdot), \mathcal P_2^\varepsilon[\textbf{W}^\varepsilon(\cdot,\xi)] \rangle | \leq |\Omega_2^\varepsilon(\xi)| |\psi|_{_{L^\infty}} \quad \forall \psi \in C(I).
\end{aligned}
       \end{equation*}  
We also assume that $\textbf{W}^\varepsilon$ is asymptotically a solution, i.e. we require  that 
\begin{equation*}
\lim_{\varepsilon \to 0} |\Omega_1^\varepsilon|_{_{L^\infty}} =0, \quad \lim_{\varepsilon \to 0} |\Omega_2^\varepsilon|_{_{L^\infty}} =0,
\end{equation*}
uniformly with respect to $\xi$.

Example \ref{ex2} show that hypothesis {\bf H1} is verified in the case of the quadratic flux $f(u)=u^2/2$.
\vskip0.2cm

{\bf H2.} There exists a constant $c_0>0$ such that
\begin{equation*}
|\Omega_1^\varepsilon(\xi)|+|\Omega^\varepsilon_2(\xi)| \leq c_0 |\lambda_{1,+}^{JX}(\xi)|, \quad \forall \ \xi  \in (-\ell,\ell).
\end{equation*} 
By comparing the asymptotic expression for $\lambda^{JX}_{1,+}$ given in \eqref{lambdaasybur} with the one for $\Omega^\varepsilon_1$ and $\Omega^\varepsilon_2$ obtained in Example \ref{ex3}, we can easily check that hypothesis {\bf H2} is verified for the Jin-Xin system when $f(u)=u^2/2$.
\vskip0.2cm
{\bf H3.} For what concern the eigenvalues of the linear operator $\mathcal L^\varepsilon_\xi$, we have proven that there exist two positive constants $c_1, c_2$ independent on $\xi$ such that
\begin{equation*}
\lambda_{1,+}^{JX}(\xi)-Re[\lambda_{2,\pm}^{JX}(\xi)]>c_1, \quad  -e^{-c_2/\varepsilon}<\lambda_{1,+}^{JX}(\xi) <0 \quad \forall \, \xi \in (-\ell,\ell).
\end{equation*}
\vskip0.2cm

{\bf H4.} Concerning the solution $Z=(z,w)^T$ to the linear problem $\partial_t Z=\mathcal L^\varepsilon_\xi Z$, we require that there exists $\nu^\varepsilon>0$ such that for all $\xi \in (-\ell,\ell)$, there exist constants $C_\xi$ and $\bar C$ such that
\begin{equation}\label{0H4}
|(z,w)(t)|_{{}_{L^2}} \leq C_\xi |(z_0,w_0)|_{{}_{L^2}}e^{-\nu^\varepsilon t}, \quad C_\xi \leq \bar C \ \ \ \forall \xi \in (-\ell,\ell).
\end{equation}
\begin{remark}\rm{
The assumption that $C_\xi< \bar C$ for all $\xi$ means that the estimate \eqref{0H4} holds uniformly in $\xi$. Since $\xi$ belongs to a bounded interval of the real line, if we suppose that $\xi \mapsto C_{\xi(t)}$ is a continuous function, then there exists a maximum $\bar C$ in $[-\ell,\ell]$. For example, in the case of the Jin-Xin system with $f(u)=u^2/2$, the constant $\nu^\varepsilon$ behaves like $|\lambda_{1,+}^{JX}(\xi)| \sim e^{-1/\varepsilon}$, and the estimate \eqref{0H4} is independent of $\xi$.
}
\end{remark}
\subsection{Estimate on the perturbation $Y$}
Our first aim is to obtain an estimate  the perturbation $Y$. We recall that
\begin{equation}\label{eqpert}
\partial_tY=(\mathcal L_\xi^\varepsilon +\mathcal M_\xi^\varepsilon)Y +H^\varepsilon(x;\xi),
\end{equation}
where
\begin{equation*}
\begin{aligned}
\mathcal M_\xi^\varepsilon Y &=-\frac{1}{\alpha_0^\varepsilon(\xi)}\left ( \begin{aligned}
\partial_\xi U^\varepsilon(\cdot;\xi) \,\theta^\varepsilon(&\xi)\,\langle \partial_\xi \boldsymbol{\psi}^\varepsilon_1(\cdot;\xi), Y \rangle\\
\partial_\xi V^\varepsilon(\cdot;\xi) \,\theta^\varepsilon(&\xi)\, \langle \partial_\xi \boldsymbol{\psi}^\varepsilon_1(\cdot;\xi), Y \rangle \end{aligned}\right), \\
H^\varepsilon(x;\xi)&=\left( \begin{aligned}
\mathcal P_1^\varepsilon[ \textbf{W}^\varepsilon(\cdot;\xi)]&-\partial_\xi U^\varepsilon(\cdot;\xi) \theta^\varepsilon(\xi) \\
 \mathcal P_2^\varepsilon[\textbf{W}^\varepsilon(\cdot,\xi)]&-\partial_\xi V^\varepsilon(\cdot;\xi) \theta^\varepsilon(\xi) \end{aligned} \right).
\end{aligned} 
\end{equation*}
In particular, $\mathcal M^\varepsilon_\xi$ is a bounded operator, such that
\begin{equation}\label{asyMH}
\begin{aligned}
\|\mathcal M^\varepsilon_\xi \|_{\mathcal L(L^2;\R^2)} \leq C|\theta^\varepsilon(\xi)| \leq C( |\Omega_1^\varepsilon|_{{}_{L^\infty}}+|\Omega_2^\varepsilon|_{{}_{L^\infty}}), \quad \forall \xi \in (-\ell,\ell). 
\end{aligned}
\end{equation}
Indeed, if we ask the family $\{ \textbf{W}^\varepsilon\}$ to be never transversal to the first eigenfunction of the corresponding linearized operator, we can assume
\begin{equation*}
|\alpha_0^\varepsilon(\xi)| = \langle \boldsymbol{\psi}^\varepsilon_1(\cdot;\xi),\partial_\xi \textbf{W}^\varepsilon(\cdot,\xi) \rangle \geq c_0 >0,
\end{equation*}
for some $c_0$ independent on $\xi$. This gives us a (weak) restriction on the choice of the family $\{ \textbf{W}^\varepsilon\}$.

Concerning the term $H^\varepsilon$, we have
\begin{equation}\label{asyMH2}
|H^\varepsilon|_{{}_{L^\infty}} \leq C_1| \Omega_1^\varepsilon|_{{}_{L^\infty}}+C_2|\Omega_2^\varepsilon|_{{}_{L^\infty}},
\end{equation}
for some positive constants $C_1$ and $C_2$ independent on $\xi$ and $\varepsilon$. 

For the special case of $f(u)=u^2/2$, both $\mathcal M^\varepsilon_\xi$ and $H^\varepsilon$ are bounded by terms that are exponentially small in $\varepsilon$, while, for a general class of flux functions $f(u)$ that verify \eqref{ipof}, the hypotheses we required assure that all the terms in the equations for the perturbation $Y$ are small in $\varepsilon$.

\begin{theorem}\label{mainthe}
Let hypotheses {\bf H1-4} be satisfied. Then, for $\varepsilon$ sufficiently small,  the solution Y to \eqref{eqpert} satisfies the estimate
\begin{equation*}
|Y|_{{}_{L^2}}(t) \leq  \left[C_1|\Omega_1^\varepsilon|_{{}_{L^\infty}}+C_2|\Omega^\varepsilon_2|_{{}_{L^\infty}}\right] \, t+e^{-\mu^\varepsilon t}|Y_0|_{L^2}, 
\end{equation*}
for some positive constants $C_1$, $C_2$ and
\begin{equation*}
\mu^\varepsilon :=\sup_{\xi}\lambda^{JX}_{1,+}(\xi)- C(|\Omega_1^\varepsilon|_{{}_{L^\infty}}+|\Omega_2^\varepsilon|_{{}_{L^\infty}})>0, \quad \mu^\varepsilon \to 0 \, \,{\rm as} \, \, \varepsilon \to 0.
\end{equation*}
\end{theorem}

\begin{proof}
Since the operator $\mathcal L^\varepsilon_\xi+\mathcal M^\varepsilon_\xi$ is a linear operator that depends on time, to obtain rigorous estimates on the solution $Y$, we need to use the theory of {\it stable families of generators}, that is a generalization of the theory of semigroups for evolution systems of the form $\partial_t u= L u$. We will use some results of \cite{Pazy83}, which have been summarized in the Appendix A. More precisely, we want to show that $\mathcal L^\varepsilon_\xi+\mathcal M^\varepsilon_\xi$ is the infinitesimal generator of a $C_0$ semigroup $\mathcal T_\xi(t,s)$.

To this aim, concerning the eigenvalues of the linear operator $\mathcal L^\varepsilon_\xi$, we know that $\lambda_{1,+}^{JX}(\xi)$ is negative and behaves like $-e^{-1/\varepsilon}$ for all $\xi \in (-\ell, \ell)$, so that $\Lambda_1^\varepsilon:=\sup_\xi \lambda_{1,+}^{JX}(\xi)$
 is such that $-e^{-1/\varepsilon} \leq \Lambda_1^\varepsilon <0$, and this estimate is independent on $t$. Hence, by using Definition \ref{def1} and Remark \ref{rem1} (see Appendix A), we know that, for $t \in [0,T]$, $\mathcal L^\varepsilon_{\xi(t)}$ is the infinitesimal generator of a $C_0$ semigroup $\mathcal S_{\xi(t)}(s)$, $s>0$. Furthermore, since \eqref{0H4} holds, we get
\begin{equation*}
\| \mathcal S_{\xi(t)}(s)\| \leq \bar C e^{-|\Lambda_1^\varepsilon|s},
\end{equation*}
so that the family $\{ \mathcal L^\varepsilon_{\xi(t)}\}_{\xi(t) \in (-\ell,\ell)}$ is stable with stability constants $M=\bar C$ and $\omega=-|\Lambda_1^\varepsilon|$. Furthermore, since
\begin{equation*}
\|\mathcal M^\varepsilon_\xi \|_{\mathcal L(L^2;\R^2)}  \leq C( |\Omega_1^\varepsilon|_{{}_{L^\infty}}+|\Omega_2^\varepsilon|_{{}_{L^\infty}}), \quad \forall \xi \in (-\ell,\ell) ,
\end{equation*}
Theorem \ref{thpaz1} (see Appendix A) states that the family $\{ \mathcal L^\varepsilon_{\xi(t)}+ \mathcal M^\varepsilon_{\xi(t)}\}_{\xi(t) \in (-\ell,\ell)}$ is stable with $M=\bar C$ and $\omega= -|\Lambda_1^\varepsilon|+C(|\Omega_1^\varepsilon|_{{}_{L^\infty}}+|\Omega_2^\varepsilon|_{{}_{L^\infty}}) <0$. 

In order to apply Theorem \ref{thpaz3} (see Appendix A), we need to check that the domain of $\mathcal L^\varepsilon_\xi+\mathcal M^\varepsilon_\xi$ does not depend on time, and this is true since $\mathcal L^\varepsilon_\xi+\mathcal M^\varepsilon_\xi$ depends on time through the function $U^\varepsilon(x;\xi(t))$, that does not appear in the higher order terms of the operator. More precisely, the principal part of the operator does not depend on $\xi(t)$. Hence, we can define $\mathcal T_\xi(t,s)$ as the {\it evolution system} of $\partial_t Y=(\mathcal L^\varepsilon_\xi+ \mathcal M^\varepsilon_\xi)Y$, so that
\begin{equation}\label{Yfamiglieevol}
Y(t)=\mathcal T_\xi(t,s)Y_0+\int_s^t \mathcal T_\xi(t,r)H^\varepsilon(x;\xi(r)) dr, \quad 0 \leq s \leq t.
\end{equation} 
Moreover, there holds
\begin{equation*}
\|\mathcal T_\xi(t,s)\| \leq \bar Ce^{-\mu^\varepsilon (t-s)}, \qquad \mu^\varepsilon := |\Lambda^\varepsilon_1|- C(|\Omega_1^\varepsilon|_{{}_{L^\infty}}+|\Omega_2^\varepsilon|_{{}_{L^\infty}})>0.
\end{equation*}
Finally, from the representation formula \eqref{Yfamiglieevol} with $s=0$, it follows 
\begin{equation}\label{finalestY}
|Y|_{{}_{L^2}}(t) \leq e^{-\mu^\varepsilon t}|Y_0|_{{}_{L^2}}+ \sup_{\xi \in I}|H^\varepsilon|_{{}_{L^\infty}}(\xi) \int_0^t e^{-\mu^\varepsilon(t-r)} \ dr,  
\end{equation}
so that, by using \eqref{asyMH2}, we end up with
\begin{equation}\label{stimafinaleY}
|Y|_{{}_{L^2}}(t) \leq  \left[C_1|\Omega_1^\varepsilon|_{{}_{L^\infty}}+C_2|\Omega^\varepsilon_2|_{{}_{L^\infty}}\right]\, t+e^{-\mu^\varepsilon t}|Y_0|_{L^2}.
\end{equation}

\end{proof}

\begin{remark}\rm{

In the special case of Burgers flux, $\mu^\varepsilon$ is going to zero exponentially as $\varepsilon \to 0$, since $\lambda_1^\varepsilon$ behaves like $e^{-1/\varepsilon}$ and from the explicit  formula of $\Omega_1^\varepsilon$ and $\Omega^\varepsilon_2$ in Example \ref{ex2}. In the general case, assumptions {\bf H1-2} assure that $\mu^\varepsilon \to 0$ as $\varepsilon\to 0$.
}
\end{remark}

\subsection{Slow motion of the shock layer}
An immediate consequence of the estimate \eqref{stimafinaleY} is that, for $|Y|_{{}_{L^2}} <M$ for some $M>0$, the function $\xi(t)$ satisfies
\begin{equation*}
\frac{d\xi}{dt}=\theta^\varepsilon(\xi)(1+r) \quad {\rm with} \quad |r| \leq \left[ C_1|\Omega_1^\varepsilon|_{{}_{L^\infty}}+C_2|\Omega^\varepsilon_2|_{{}_{L^\infty}} \right] t +e^{-\mu^\varepsilon t}|Y_0|_{{}_{L^2}}.
\end{equation*}
More precisely, we can prove the following Proposition.

\begin{proposition}\label{SLBfin}
Let hypotheses {\bf H1-4} be satisfied. Assume also 
     \begin{equation}\label{altripth2}
(\xi-\bar \xi) \, \theta^\varepsilon(\xi)<0 \quad \textit{for any \, } \xi \in I, \, \xi \neq 0 \qquad \textit{and} \qquad { \theta^\varepsilon}'(\bar \xi) <0. 
       \end{equation}
Then, for  $\varepsilon$ and $|Y_0|_{{}_{L^2}}$ sufficiently small, the solution $\xi$ converges to $\bar \xi$ as $t \to +\infty$.
\end{proposition}       

 \begin{proof}    

Due to the estimate \eqref{stimafinaleY}, for $\varepsilon$ and $|Y_0|_{{}_{L^2}}$ sufficiently small and  for any initial datum $\xi_0$, the location of the shock layer satisfies 
\begin{equation}\label{estimateoneta}
\int_{\xi_0}^{\xi(t)} \frac{dz}{\theta^\varepsilon(z)}=\int_0^t(1+r(s))ds,
\end{equation}
where
\begin{equation*}
 |r(t)| \leq \left[ C_1|\Omega_1^\varepsilon|_{{}_{L^\infty}}+C_2|\Omega^\varepsilon_2|_{{}_{L^\infty}} \right] t+e^{-\mu^\varepsilon t}|Y_0|_{L^2}.
\end{equation*}
More precisely,  in the regime of small $\varepsilon$, the shock location $\xi(t)$ has similar decays properties to those of the solution to the following reduced problem
\begin{equation}\label{eqxiapprossi}
\frac{d\eta}{dt}=\theta^\varepsilon(\eta), \quad \eta(0)=\xi(0) \ ,\qquad {\rm with} \,\,\, \theta^\varepsilon(\eta)=\frac{\langle \boldsymbol \psi^\varepsilon_1,\mathcal F[\textbf{W}^\varepsilon] \rangle}{\langle \boldsymbol \psi^\varepsilon_1, \partial_\eta \textbf{W}^\varepsilon \rangle}.
\end{equation}
By means of a standard method of separation of variable, we get
\begin{equation*}
\int_{\xi_0}^{\xi} \frac{d\xi}{\theta^\varepsilon(\xi)} = \int_0^t \, dt. 
\end{equation*}
Since $\theta^\varepsilon(\xi) \sim{ \theta^\varepsilon}'(\bar\xi)(\xi -\bar \xi) $, by integrating we obtain the following estimate for the shock layer location 
\begin{equation}\label{metastabxi}
|\xi(t)-\bar \xi| \leq |\xi_0|e^{\beta^\varepsilon t}, \quad \beta^\varepsilon \sim{ \theta^\varepsilon}'(\bar\xi),
\end{equation}
where $\bar \xi$ represent the equilibrium location for the shock layer position and ${ \theta^\varepsilon}'(\bar \xi) \to 0$ as $\varepsilon \to 0$. Therefore $\xi$ converges to $\bar \xi$ as $t \to +\infty$, and  the convergence is exponential for any $t$ under consideration.

\end{proof} 

Formula \eqref{metastabxi} shows the slow motion of the shock layer for small $\varepsilon$. Precisely, the evolution of the collocation of the shock towards the equilibrium position is much slower as $\varepsilon$ becomes smaller. 

For example, when $f(u)=u^2/2$, $\bar \xi =0$ and ${ \theta^\varepsilon}'(0)\sim -e^{-1/\varepsilon}$  (see formula \eqref{stimasutheta}).  We also emphasize that hypotheses \eqref{altripth2} are verified in the case of the Jin-Xin system with $f(u)=u^2/2$.

The following table shows a numerical computation for the location of the shock layer for different values of the parameter $\varepsilon$ and $f(u)=u^2/2$. The initial datum for the function $u$ is $u_0(x)=\frac{1}{2}x^2-x-\frac{1}{2}$. We can see that the convergence to $\bar \xi =0$ is slower as $\varepsilon$ becomes smaller.

\begin{table}[h!]
\footnotesize{The numerical location of the shock layer $\xi(t)$ for different values of the parameter $\varepsilon$}
%\footnotesize{ 
\begin{center}
\begin{tabular}{|c|c|c|c|c|c|}
\hline TIME $t$ &   $\xi(t)$, $\varepsilon=0.1$  &   $\xi(t)$, $\varepsilon=0.07$ &  $\xi(t)$, $\varepsilon=0.055$ &   $\xi(t)$, $\varepsilon=0.04$ &  $\xi(t)$, $\varepsilon=0.02$
\\ \hline
\hline $0.2$ Ê& $-0.4008$ &$-0.4020$  &$-0.4029$ &$-0.4040$ &$-0.4059$\\
\hline  $1$ & $-0.3314$ & $-0.3345$  &$-0.3360$ &$-0.3374$ & $-0.3389$\\
\hline $10$ & $-0.3070$ &$-0.3263$ &$-0.3304$  &$-0.3320$ & $-0.3326$\\
\hline $10^3$ & $-0.0103$ &$-0.1600$  &$-0.2562$ &$-0.3181$ & $-0.3325$\\
\hline $10^4$ & $-1.9725*10^{-12}$ &$-0.0084$ &$-0.1115$ &$-0.2531$& $-0.3320$\\
\hline $0.5*10^6$ & $-1.9725*10^{-12}$ &$-2.2102*10^{-11}$ &$-1.5057*10^{-10}$ &$-0.0379$ & $-0.3099$\\
\hline\end{tabular}
\end{center}

\end{table}

Figure \ref{fig4} shows the dynamics of the shock layer (i.e the dynamics of the solution $u$ to \eqref{JX2}), obtained numerically. When $\varepsilon=0.1$, the shock layer location converges to zero very fast: as we can also see from the table, when $t=10^3$, the value of $\xi(t)$ is already very close to zero. On the other hand, when $\varepsilon$ becomes smaller the shock layer location moves slower and it approaches the equilibrium location only for very large $t$.
Finally, Figure \ref{fig5} shows the profile of the shock layer for the flux function $f(u)=u^4/4$, that still verifies hypotheses \eqref{ipof}.

\begin{figure}
\centering
\includegraphics[width=7cm,height=7cm]{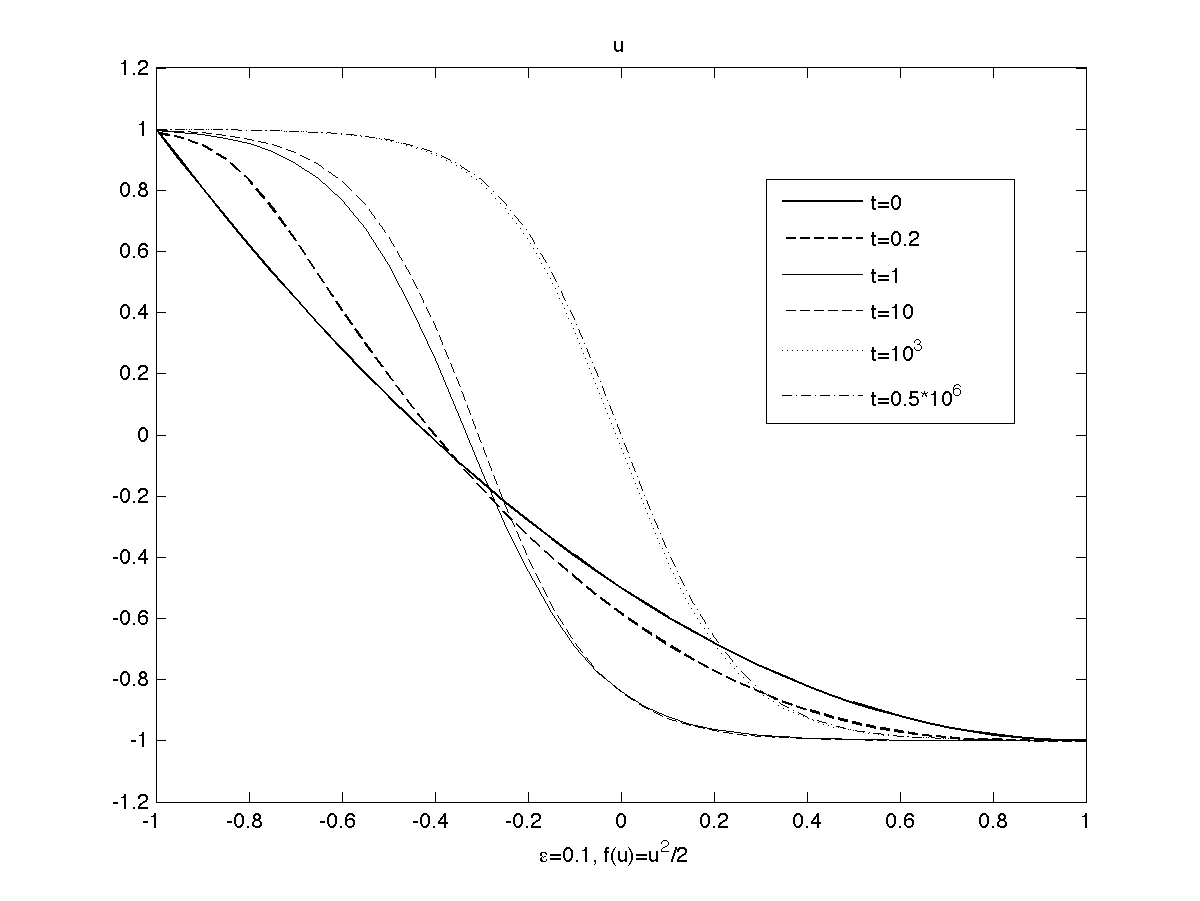}
 \hspace{3mm}
\includegraphics[width=7cm,height=7cm]{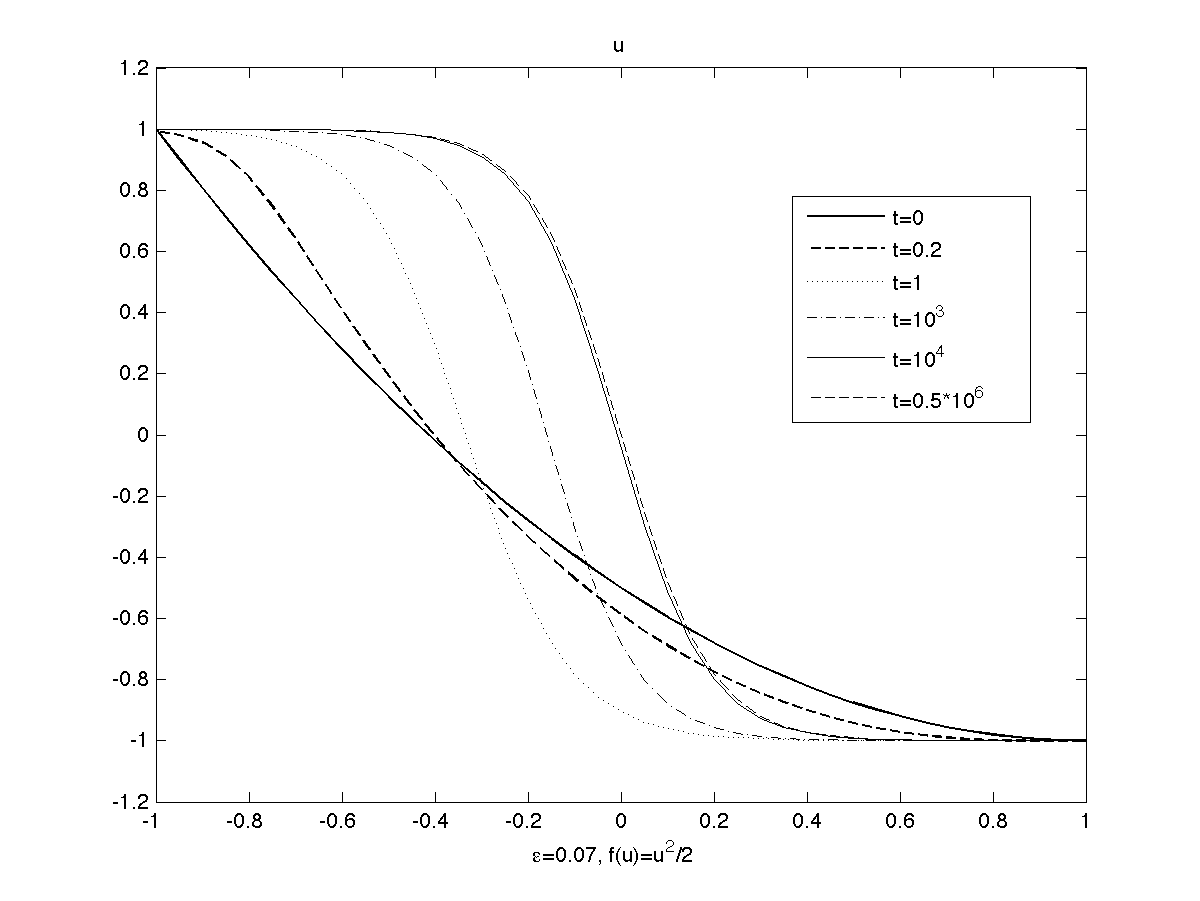}f
 \hspace{3mm}
\includegraphics[width=7cm,height=7cm]{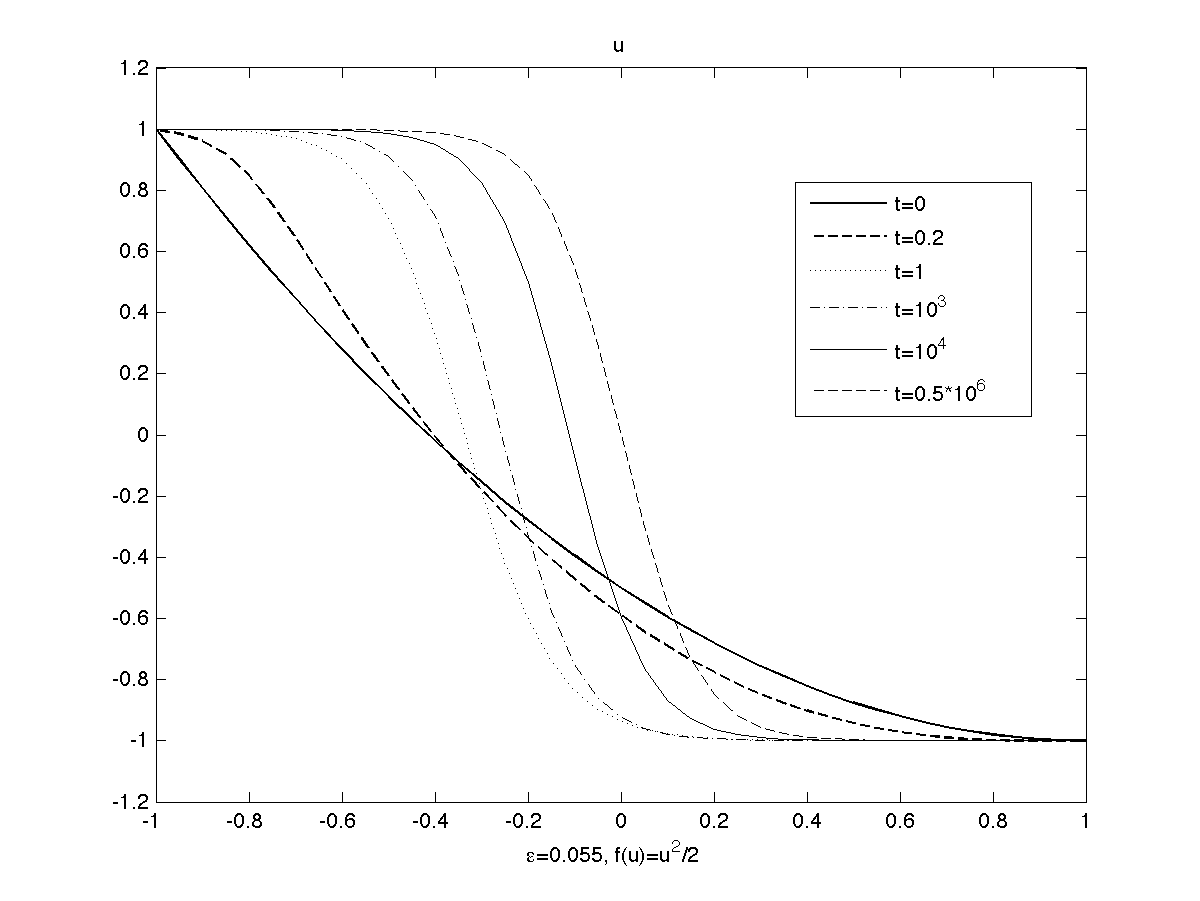}
 \hspace{3mm}
\includegraphics[width=7cm,height=7cm]{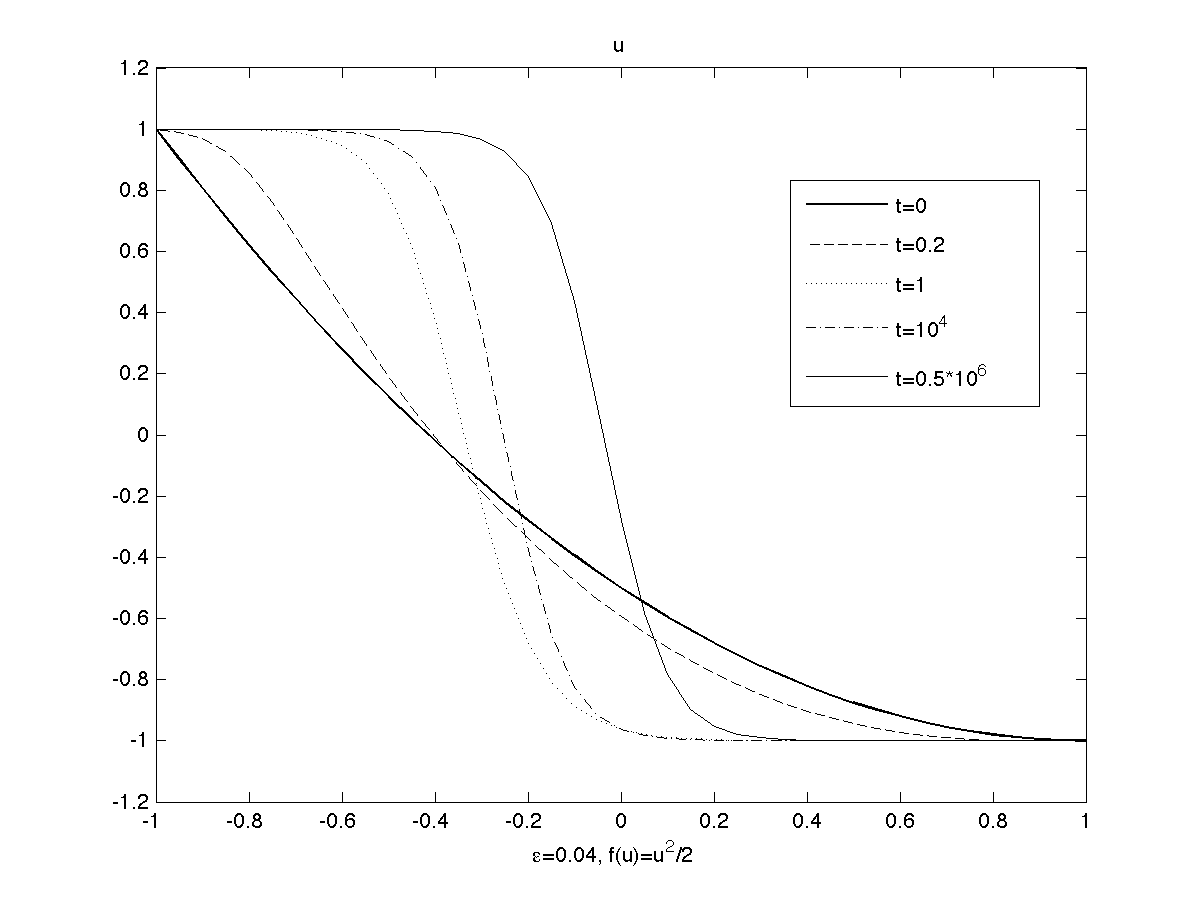}
\caption{\small{The shock layer profiles for different times and different values of the parameter $\varepsilon$. }}\label{fig4}
\end{figure}

\begin{figure}[ht]
\centering
\includegraphics[width=16cm,height=13cm]{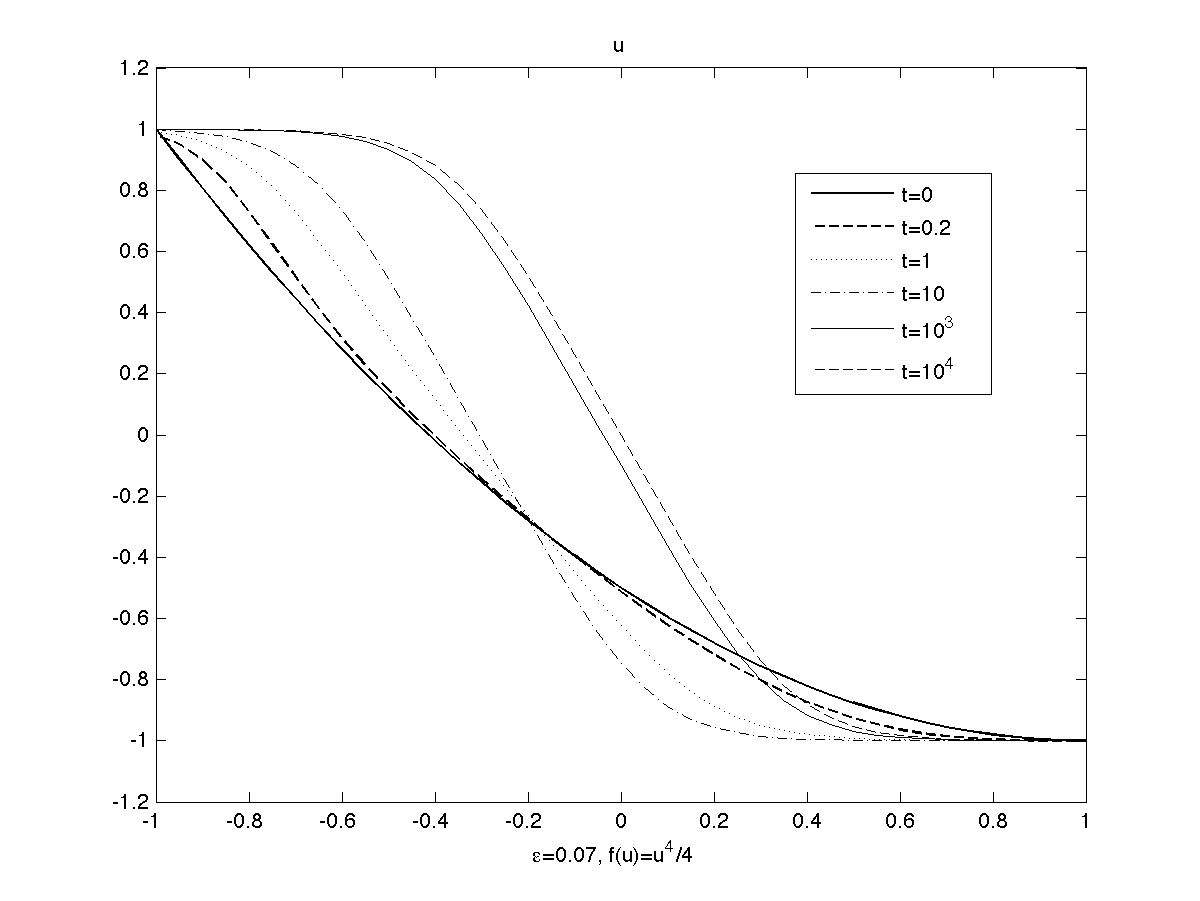}
\caption{\small{Profiles of the shock layer at different times with the convex flux function $f(u)=u^4/4$.}}\label{fig5}
\end{figure}

\section{Appendix A}

In this section, we briefly review some results on the theory of evolution systems by A. Pazy \cite[Chapter 5]{Pazy83}. For more details and for the proofs of the Theorems, see \cite[Theorem 2.3, Theorem 3.1, Theorem 4.2]{Pazy83}.

Let $X$ be a Banach space. For every $0 \leq t \leq T$, let $A(t): D(A(t)) \subset X \to X$ be a linear operator in $X$ and let $f(t)$ be an $X$ valued function. Let us consider the initial value problem 
\begin{equation}\label{pazy}
\partial_t u = A(t) u+ f(t), \quad u(s)=u_0, \qquad 0 \leq s \leq t \leq T.
\end{equation}
In the special case where $A(t)=A$ is independent of $t$, the solution to \eqref{pazy} can be represented via the formula of {\it variations of constants}
\begin{equation*}
u(t)= T(t)u_0+ \int_0^t T(t-s)f(s) \ ds
\end{equation*}
where $T(t)$ is the $C_0$ semigroup generated by $A$. In \cite{Pazy83} it is shown that a similar representation formula is true also when $A(t)$ depends on time.
\begin{ans}\label{def1}
Let $X$ a Banach space. A family $\{ A(t)\}_{t \in [0,T]}$ of infinitesimal generators of $C_0$ semigroups on $X$ is called stable if there are constants $M \geq 1$ and $\omega$ (called the stability constants) such that
\begin{equation*}
(\omega,+\infty) \subset \rho(A(t)) \quad {\rm for} \quad t \in [0,T]
\end{equation*}
and
\begin{equation*}
\left \| \Pi_{j=1}^k R(\lambda: A(t_j))\right \|  \leq M(\lambda-\omega)^{-k}, \quad {\rm} 
\end{equation*}
for $\lambda>\omega$ and for every finite sequence $0 \leq t_1 \leq t_2, . . . . , t_k \leq T$, $k=1,2,....$.
\end{ans}
\begin{remark}\label{rem1}{\rm
If for $t \in [0,T]$, $A(t)$ is the infinitesimal generator of a $C_0$ semigroup $S_t(s)$, $s \geq 0$ satisfying $\| S_t(s)\| \leq e^{\omega s}$, then the family $\{ A(t)\}_{t \in [0,T]}$ is clearly stable with constants $M=1$ and $\omega$.
}
\end{remark}

The previous remark means that, if for every fixed $t \in [0,T]$ the operator $A(t)$ generates a $C_0$ semigroup $S_t(s)$, and we can find an estimate for $\| S_t(s)\|$ that is independent of $t$, then the whole family $\{ A(t)\}_{t \in [0,T]}$ is stable in the sense of Definition \ref{def1}.

\begin{theorem}\label{thpaz1}
Let $\{ A(t)\}_{t \in [0,T]}$ be a stable family of infinitesimal generators with stability constants $M$ and $\omega$. Let $B(t)$, $0 \leq t \leq T$ be a bounded linear operators on $X$. If $\| B(t)\| \leq K$ for all $t \leq T$, then $\{ A(t)+ B(t)\}_{t \in [0,T]}$ is a stable family of infinitesimal generators with stability constants $M$ and $\omega+ MK$.
\end{theorem}

In order to prove the existence of the so called {\it evolution system} $U(t,s)$ for the initial value problem \eqref{pazy}, let us introduce two Banach spaces $X$ and $Y$, with norms $\| \ \|_X$, $\| \ \|_Y$ respectively. Moreover, let us assume that $Y$ is a dense subspace of $X$ and that there exists a constant $C$ such that $\|w \|_X \leq C \| w\|_Y$ for all $w \in Y$. 
\begin{ans}
Let $A$ be the infinitesimal generator of a $C_0$ semigroup $S(s)$, $s \geq 0$, on $X$. $Y$ is called $A$-{\it admissible} if it is an invariant subspace of $S(s)$, and the restriction $\tilde S(s)$ of $S(s)$ to $Y$ is a $C_0$ semigroup on $Y$. Moreover, the infinitesimal generator of the semigroup $\tilde S(s)$ on $Y$, denoted here with  $\tilde A$, is called the {\it part of $A$ in $Y$}.
\end{ans}

Next, let us fix $t \in [0,T]$, and let $A(t)$ be the infinitesimal generator of a $C_0$ semigroup $S_t(s)$ on $X$. The following assumptions are made
  \vskip0.2cm
 (\textbf{H1}) $\{ A(t)\}_{t \in [0,T]}$ is a stable family with stability constants $M$ and $\omega$.
 \vskip0.2cm
 (\textbf{H2}) Y is $A(t)$-admissible for $t\in[0,T]$ and the family $\{ \tilde A(t)\}_{t \in [0,T]}$ is a stable family in $Y$ with stability constants $\tilde M$, $\tilde \omega$.
 \vskip0.2cm
 (\textbf{H3}) For $t \in [0,T]$, $Y \subset D(A(t))$, $A(t)$ is a bounded operator from $Y$ into $X$ and $t \to A(t)$ in continuous in the $B(X,Y)$ norm.
 
 \begin{remark}{ \rm
 The assumption that the family $\{ A(t)\}_{t \in[0,T]}$ satisfies (\textbf{H2}) is not always easy to check. A sufficient condition for (\textbf{H2}) which can be effectively checked in many applications states that (\textbf{H2}) holds if there is a family $\{ Q(t)\}$ of isomorphisms of $Y$ onto $X$ such that $\| Q(t) \|_{Y \to X}$ and $\| Q(t)^{-1} \|_{Y \to X}$ are uniformly bounded and $t \to Q(t)$ is of bounded variation in the $B(Y,X)$ norm (for more details, see \cite[Chapter 5]{Pazy83}).
 
 }
 \end{remark}
 
 \begin{remark}{\rm
 Condition  (\textbf{H3}) can be replaced by the weaker condition
 \vskip0.2cm
  (\textbf{H3})' For $t\in[0,T]$, $Y \subset D(A(t))$ and $A(t) \in L^1([0,T];B(Y,X))$.
 }
 \end{remark}

\begin{theorem}\label{thpaz2} 
Let $A(t)$, $0 \leq t \leq T$ be the infinitesimal generator of a $C_0$ semigroup $S_t(s)$, $s \geq 0$ on $X$. If the family $\{ A(t)\}_{t \in [0,T]}$ satisfies the conditions
(\textbf{H1})-(\textbf{H3}), then there exists a unique {\it evolution system} $U(t,s)$, $0 \leq s \leq t \leq T$, in $X$ satisfying
 \begin{equation}\label{01}
 \| U(t,s) \| \leq Me^{\omega (t-s)}, \quad {\rm for } \quad 0 \leq s \leq t \leq T.
 \end{equation}
 Moreover, if $f \in C([s,T],X)$, the solution to \eqref{pazy} can be written as
 \begin{equation}\label{02}
 u(t)= U(t,s)u_0 + \int_s^t U(t,r) f(r) \ dr,
 \end{equation}
for all $0 \leq s \leq t \leq T$.
\end{theorem}

One special case in which the conditions of Theorem \ref{thpaz2} can be easily checked is the case where the domain of the operator $D(A(t)) \equiv D$ is independent on $t$. In this case we can take $D$ as the Banach space which we denote by $Y$, and the following Theorem holds
\begin{theorem}\label{thpaz3}
Let $\{ A(t)\}_{t \in [0,T]}$ be a stable family of infinitesimal generators of $C_0$ semigroups on $X$. If $D(A(t))=D$ is independent on $t$ and for $u_0 \in D$, $A(t)u_0$ is continuously differentiable in $X$, then there exists a unique evolution system $U(t,s)$, $0 \leq s \leq t \leq T$, satisfying \eqref{01}. Morevoer, if $f \in C([s,T],X)$, then, for every $u_0 \in D$, the initial value problem \eqref{pazy} has a unique solution given by \eqref{02}.
\end{theorem}
\vskip1cm
\section*{Acknowledgments}

I wish to thank C. Mascia  for having introduced me to the problem and for guidance throughout writing the paper.

\end{document}